\documentclass[11pt,twoside]{article}
\usepackage{amsmath,amssymb,amsthm}
\usepackage{graphics,graphicx}
\usepackage{epstopdf}
\usepackage{lipsum}
\usepackage{braket,amsfonts,amsopn}
\usepackage{subfigure}
\usepackage{algorithm}
\usepackage{algorithmic}
\usepackage{cases}
\usepackage{color}
\usepackage{fancyhdr}
\usepackage{empheq}
\usepackage{marginnote}
\usepackage{optidef}
\usepackage{float}
\usepackage{multicol}
\usepackage{multirow}
\usepackage{hyperref}
\usepackage{blkarray}
\usepackage{enumerate}
\usepackage{verbatim}
\usepackage{optidef}
\usepackage[asymmetric]{geometry}
\usepackage{cancel}

\ifpdf
\DeclareGraphicsExtensions{.eps,.pdf,.png,.jpg}
\else
\DeclareGraphicsExtensions{.eps}
\fi

\newcommand\AMSname{AMS subject classifications}

\newcommand\keywordsname{Key words}

\newcommand{\newcomment}[1]{}

\newenvironment{@abssec}[1]{%
	\if@twocolumn
	\section*{#1}%
	\else
	\vspace{.05in}\footnotesize
	\parindent .2in
	{\upshape\bfseries #1. }\ignorespaces
	\fi}
{\if@twocolumn\else\par\vspace{.1in}\fi}

\newenvironment{AMS}{\begin{@abssec}{\AMSname}}{\end{@abssec}}

\newtheorem{example}{Example}
\newtheorem{definition}{Definition}[section]

\numberwithin{equation}{section}
\numberwithin{figure}{section}
\numberwithin{table}{section}

\newtheorem{theorem}{Theorem}[section]

\newtheorem{lemma}{Lemma}[section]
\newtheorem{proposition}{Proposition}[section]

\newcommand{\email}[1]{\protect\href{mailto:#1}{#1}}

\floatname{algorithm}{Algorithm}

\newenvironment{keywords}{\begin{@abssec}{\keywordsname}}{\end{@abssec}}

\newcommand{\mat}[1]{\left[ \begin{array}{#1} }
\newcommand{\rix}{\end{array} \right]}

\newcommand{\twobytwo}[4]{
       \left[ \begin{array}{cc}
        #1 & #2  \\
        #3 & #4
           \end{array} \right] }

\fancypagestyle{plain}{
	\fancyhf{}
	\cfoot{\thepage}
	
}

\makeatletter
\def\bbordermatrix#1{\begingroup \m@th
	\@tempdima 4.75\p@
	\setbox\z@\vbox{%
		\def\cr{\crcr\noalign{\kern2\p@\global\let\cr\endline}}%
		\ialign{$##$\hfil\kern2\p@\kern\@tempdima&\thinspace\hfil$##$\hfil
			&&\quad\hfil$##$\hfil\crcr
			\omit\strut\hfil\crcr\noalign{\kern-\baselineskip}%
			#1\crcr\omit\strut\cr}}%
	\setbox\tw@\vbox{\unvcopy\z@\global\setbox\@ne\lastbox}%
	\setbox\tw@\hbox{\unhbox\@ne\unskip\global\setbox\@ne\lastbox}%
	\setbox\tw@\hbox{$\kern\wd\@ne\kern-\@tempdima\left[\kern-\wd\@ne
		\global\setbox\@ne\vbox{\box\@ne\kern2\p@}%
		\vcenter{\kern-\ht\@ne\unvbox\z@\kern-\baselineskip}\,\right]$}%
	\null\;\vbox{\kern\ht\@ne\box\tw@}\endgroup}
\makeatother

\title{2D Eigenvalue Problem II: Rayleigh Quotient Iteration and Applications~\thanks{
version \today.
}
}

\author{
Tianyi Lu\thanks{School of Mathematical Sciences, 
Fudan University, Shanghai 200433, China
(\email{tylu17@fudan.edu.cn}, \email{yfsu@fudan.edu.cn}). 
}
\and Yangfeng Su\footnotemark[2]
\and Zhaojun Bai\thanks{Department of Computer Science
and Department of Mathematics, University of California,
Davis, CA 95616, USA (\email{zbai@ucdavis.edu})}
}

\usepackage{amsopn}
\DeclareMathOperator{\diag}{diag}

\DeclareMathOperator{\imag}{Imag}
\DeclareMathOperator{\Span}{span}
\DeclareMathOperator{\sign}{sign}
\DeclareMathOperator{\dist}{dist}

\DeclareMathOperator{\vvec}{vec}

\ifpdf
\hypersetup{
pdftitle={2D Eigenvalue Problems II},
pdfauthor={T. Lu, Y. Su and Z. Bai}
}
\fi

\begin{document}
\maketitle

\begin{abstract}
In Part I of this paper, we 
introduced a 2D eigenvalue problem (2DEVP) and presented
theoretical results of the 2DEVP and its intrinsic connetion with the 
eigenvalue optimizations. 
In this part, we devise a Rayleigh quotient iteration (RQI)-like 
algorithm, 2DRQI in short, 
for computing a 2D-eigentriplet of the 2DEVP. 
The 2DRQI performs $2\times$ to $8\times$  
faster than the existing algorithms for large scale eigenvalue 
optimizations arising from the minmax of Rayleigh quotients and 
the distance to instability of a stable matrix. 
\end{abstract}




\section{Introduction}\label{eq:intro}

This is the second part of the paper in the sequel on 
the 2D eigenvalue problem (2DEVP), namely computing 
scalars $\mu, \lambda \in \mathbb{R}$ and nonzero vector 
$x \in \mathbb{C}^n$ such that
\begin{subequations}\label{2deig}
\begin{empheq}[left={}\empheqlbrace]{alignat=2}
(A-\mu C) x & = \lambda x,  \label{eq:1a} \\
x^HCx & = 0, \label{eq:1b}   \\
x^Hx & = 1,    \label{eq:1c}
\end{empheq}
\end{subequations}
where $A, C \in \mathbb{C}^{n\times n}$ are 
given Hermitian matrices and $C$ is indefinite. 
The pair $(\mu, \lambda)$ is called a \emph{2D-eigenvalue},
$x$ is called the corresponding \emph{2D-eigenvector}, and
the triplet $(\mu,\lambda,x)$ is called a \emph{2D-eigentriplet}.

In Part I \cite{2DEVPI}, 
we presented the theory of the 2DEVP~\eqref{2deig}, 
such as association with the parameter eigenvalue problem 
and existence and variational characterization of 2D-eigenvalues.
We revealed that the 2DEVP has intrinsic relation with 
the problem of eigenvalue optimization.
Specifically, the equation \eqref{eq:1a} is a parameter eigenvalue problem
of $H(\mu) = A - \mu C$. Since $A$ and $C$ are Hermitian,
$H(\mu)$ has $n$ real eigenvalues 
$\lambda_1(\mu), \lambda_2(\mu), \ldots, \lambda_n(\mu)$
for any $\mu\in\mathbb{R}$.
Suppose these eigenvalues are sorted such that
$\lambda_1(\mu) \ge \lambda_2(\mu) \ge \cdots \ge \lambda_n(\mu)$, 
then equation~\eqref{eq:1b}
is a necessary condition for (local or global) maxima or minima of
$\lambda_i(\mu)$.

In this paper, we focus on numerical algorithms for solving
the 2DEVP \eqref{2deig}.  Rayleigh quotient iteration (RQI) is 
a classical and efficient algorithm for 
computing an eigenpair of an Hermitian matrix, 
see \cite{Parlett1987Symmetric,Tapia2018} and references therein. 
The RQI is locally cubically convergent, i.e., the number of 
correct digits triples at each iteration once the error 
is small enough and the eigenvalue is 
simple~\cite[Theorem~5.9]{demmel1997applied}.
In this paper, we devise an RQI-like algorithm called 2DRQI 
for solving the 2DEVP \eqref{2deig}.
One of main advantages of the 2DRQI is that the 
computational kernel of the 2DRQI is 
a linear systems of equation, similar to the classical RQI. 
Therefore, the 2DRQI is capable to solve 
large scale 2DEVP by exploiting the structure and sparsity of
matrices $A$ and $C$. 

As a part of main contributions of this part, the 2DRQI 
is further developed for applications in two eigenvalue 
optimization problems, namely  
finding the minmax of two Rayleigh quotients and 
computing the distance to instability (DTI) of a stable matrix. 
We will demonstrate the algorithmic advantages of treating
these eigenvalue optimizations through the 2DEVP and the 2DRQI, such
as introducing the notion of the backward error of
a computed DTI for the first time and 
the significant reduction ($2\times$ to $8\times$ speedups)  
in computing time comparing with the existing algorithms.  

In the third part of this paper, we will provide
a rigorous convergence analysis of the proposed 2DRQI, 
and prove that the 2DRQI is locally quadratically convergent 
under some mild assumptions. 

The rest of this paper is organized as follows.
In Section~\ref{sec:2drq}, we will introduce 
concept of 2D Rayleigh quotients (2DRQ) and Jacobian of the 2DEVP, 
and present the approximation properties of the 2DRQ. 
In Section~\ref{sec:rqi}, we derive a 2D Rayleigh quotient iteration 
(2DRQI). The backward error analysis of the 2DEVP for an approximate
2D-eigentriplet is in Section~\ref{sec:backerr}. 
Section~\ref{sec:apps} discusses the applications of the
2DRQI for finding the minmax of two Rayleigh quotients and 
computing the distance to instability (DTI) of a stable matrix. 
In Section~\ref{sec:Experiments}, we present numerical examples
to illustrate the convergence behavior of the 2DRQI and 
demonstrate its efficiency for the applications.
Concluding remarks are in Section~\ref{sec:conclusion}. 

%
\section{2D Rayleigh quotient} \label{sec:2drq}

In this section, we first introduce the concepts of
Rayleigh quotient and Ritz values for the 2DEVP~\eqref{2deig},
and then reveal their approximation property to 2D-eigentriplets.

\begin{definition}\label{def:2DRQ}
Given an $n\times n$ Hermitian matrix pair $(A,C)$ and
an $n\times p$ orthonormal matrix $V$, 
the $p\times p$ matrix pair $(V^HAV,V^HCV)$ is 
called a {\em 2D Rayleigh quotient (2DRQ)}. 

Let $(\nu,\theta,z)$ be a 2D-eigentriplet of
the 2DRQ $(V^HAV,V^HCV)$ when $V^HCV$ is indefinite, 
i.e.,
\begin{subequations}\label{eq:reduced2devp}
\begin{empheq}[left={}\empheqlbrace]{alignat=2}
\Big( (V^HAV) - \nu (V^HCV)\Big)z & = \theta z, \label{eq:red2devp1} \\
z^H(V^HCV)z & = 0,  \label{eq:red2devp2}\\
z^Hz & = 1, \label{eq:red2devp3}
\end{empheq}
\end{subequations}
then $(\nu,\theta)$ is called a {\em 2D Ritz value}, 
$Vz$ a {\em 2D Ritz vector}, 
and $(\nu,\theta,Vz)$ a {\em 2D Ritz triplet}.
\end{definition}

The pair $(V^HAV,V^HCV)$ is called a 2DRQ for two reasons. 
First, it is analogous to the definition~\cite[p.\,288]{Parlett1987Symmetric} 
of the RQ for one matrix.  Second, in Section~\ref{subsec:outline},
we will see that when $C=0$, a Rayleigh quotient iteration (RQI) like 
method to solve the 2DEVP~\eqref{2deig} degenerates to the well-known
RQI for an eigenpair of a Hermitian matrix~\cite[Sec.~4.6]{Parlett1987Symmetric}
and \cite{Tapia2018}. 

The 2DEVP \eqref{2deig} can be formulated as the problem of finding 
the root of the following system of nonlinear equations
\[
F(\mu,\lambda,x) \equiv
\left[ \begin{array}{r}
	Ax-\mu Cx-\lambda x  \\
	-x^H C x/2  \\
	-(x^H x - 1)/2 \\
\end{array} \right]  = 0.
\]
When $\mu, \lambda$ and $x$ are real, the Jacobian of the function $F$ 
is well defined, see e.g. \cite[p.65]{Kelley1995iterative}.
When $x$ is complex, the second and third elements of $F$ 
are not differentiable due to the violation of Cauchy-Riemann 
conditions~\cite{Knopp1990theoryI}. 
In this case we have the following natural extension of the Jacobian of 
the nonlinear function $F$.

\begin{definition} \label{def:jacobian} 
The {\em Jacobian} of $F$ (and the 2DEVP) is defined as 
\begin{equation} \label{eq:jacobi}
J(\mu,\lambda,x)=
\left[
\begin{array}{c|cc}
A-\mu C-\lambda I & -Cx &-x \\
\hline
-x^HC & 0 & 0\\
-x^H & 0 & 0
\end{array}
\right].
\end{equation}
\end{definition} 


For an $n\times 2$ orthonormal matrix $V$ of 
certain properties, the following theorem shows that 
if a 2D-eigenvector is near 
the subspace spanned by $V$, then 
the 2D Ritz triplet induced by $V$ will contain
a good approximation to a 2D-eigentriplet. A proof of the theorem will be provided in \cite{2DEVPIII}.

\begin{theorem}\label{Thm2DRQI} 
Let $(\mu_*,\lambda_*,x_*)$ be a 2D-eigentriplet of $(A,C)$. 
For any $\gamma > 0$, 
denote $\mathcal{V}_{\gamma}$ as the set of 
$n\times2$ orthonormal matrices $V$ satisfying 
$V^HCV$ is diagonal, $\det(V^HCV) \leq -\gamma$, and
$|(V^HAV)_{12}| \geq \gamma$. 
Then there exists constants $\alpha_1,\alpha_2,\alpha_3$ only depending on $(A, C, \mu_*, \lambda_*,\gamma)$, such that for any $V\in\mathcal{V}_{\gamma}$, let 
$\epsilon = \dist(x_*,\Span\{V\})\equiv\min\limits_{v\in\Span\{V\} }\|x-v\|_2$ and assume $\epsilon<1$, there exists a 2D Ritz triplet $(\nu,\theta, Vz)$ that satisfies 
\begin{equation}\label{eq:appealing}
|\nu-\mu_*| \leq \alpha_1 \epsilon,  \quad
|\theta -\lambda_*| \leq \alpha_2 \epsilon^2  
\quad \mbox{and} \quad
\|Vz-x_*\| \leq \alpha_3 \epsilon.
\end{equation}
\end{theorem}

Theorem~\ref{Thm2DRQI} indicates that to solve the 2DEVP, 
we should first search for a subspace $V$ where a good 
approximation of a 2D-eigentriplet lies in.
This is the essential idea guiding the derivation of 
a 2D Rayleigh quotient iteration in next section. 

%
%


\section{2D Rayleigh Quotient Iteration}\label{sec:rqi}
The Rayleigh quotient iteration (RQI) is an efficient single-vector
iterative method for solving the symmetric eigenvalue 
problem~\cite[Section~4.6]{Parlett1987Symmetric},\cite{Tapia2018}.
In this section, we derive a RQI-like method to solve 
the 2DEVP~\eqref{2deig}.

Theorem~\ref{Thm2DRQI} indicates that 
the gist of a RQI-like algorithm is how to use the 
$k$th approximation $(\mu_{k},\lambda_{k},x_{k})$ 
of a 2D-eigentriplet $(\mu_*,\lambda_*,x_*)$ 
to obtain a projection subspace $V_k$ closer to a 2D-eigenvector $x_*$ 
and then define the $k+1$-st approximation $(\mu_{k+1},\lambda_{k+1},x_{k+1})$
using a 2D Ritz triplet. 

	To that end, assume 
the Jacobian $J(\mu_k,\lambda_k,x_k)$ defined in \eqref{eq:jacobi} 
is nonsingular. Write 
$$
\mu_*=\mu_k+\Delta\mu_k, \quad 
\lambda_*=\lambda_k+\Delta\lambda_k,  \quad 
x_*=x_k+\Delta x_k,
$$ 
where $|\Delta\mu_k|\leq\epsilon$, $|\Delta\lambda_k|\leq\epsilon$ and 
$\|\Delta x_k\|\leq\epsilon$ for some small $\epsilon>0$. 
Then by \eqref{eq:1a}, we have
\begin{equation}\label{eq:Jkhat}
	\begin{bmatrix} A-\mu_kC-\lambda_kI & -Cx_k & -x_k 
	\end{bmatrix} 
	\begin{bmatrix} x_* \\ \Delta\mu_k\\ \Delta\lambda_k \end{bmatrix}
	\equiv 
	\widehat{J}_k 
	\begin{bmatrix} x_* \\ \Delta\mu_k\\ \Delta\lambda_k \end{bmatrix}
	= O(\epsilon^2),
\end{equation}
This 
implies that up to the 
second-order approximation of $\epsilon$, the vector 
$\left[\begin{smallmatrix}x_*\\ 
	\Delta\mu_k\\ \Delta\lambda_k\end{smallmatrix}\right]$ 
lies in the null subspace of $\widehat{J}_k$. 
Since the Jacobian $J(\mu_k,\lambda_k,x_k)$ is assumed to be nonsingular, 
$\widehat{J}_k$ is of full rank and the dimension of the null subspace 
of $\widehat{J}_k$ is 2. 
Let $\left[\begin{smallmatrix}\widetilde{V}_k\\ R\end{smallmatrix}\right]$ 
be a basis matrix of the null subspace of $\widehat{J}_k$, 
where $\widetilde{V}_k\in\mathbb{C}^{n\times 2}$, 
$R\in\mathbb{C}^{2\times 2}$. Then by \eqref{eq:Jkhat}, 
up to the second-order approximation of $\epsilon$, $x_*$ lies approximately in 
$\Span\{\widetilde{V}_k\}$. Thus a natural idea is to use the 2D-Ritz triplet 
based on the Rayleigh quotient induced by 
$\Span\{\widetilde{V}_k\}$ to define the next iterate 
$(\mu_{k+1},\lambda_{k+1},x_{k+1})$.

To compute $\widetilde{V}_k$, one can apply the traditional methods for computing the null space of $\widehat{J}_k$, such as the rank revealing QR decomposition~\cite[p.107]{demmel1997applied}. However, for exploiting the underlying structure and sparsity of $(A,C)$, we consider the following augmented linear equation of \eqref{eq:Jkhat}: 
\begin{equation}\label{2DRQIaugmen}
	J(\mu_k, \lambda_k, x_k) 
	\begin{bmatrix}
		X_a\\
		u\\
		v
	\end{bmatrix} = \begin{bmatrix}
		0&0\\
		1&0\\
		0&1\\
	\end{bmatrix}.
\end{equation}
By the first block row of \eqref{2DRQIaugmen}, 
$\Span\{X_a\} \subseteq \Span\{\widetilde{V}_k\}$.
Meanwhile, by the second and third block rows of \eqref{2DRQIaugmen}, 
$\dim(\Span\{X_a\})=2$. 
Since $\mbox{dim}(\Span\{\widetilde{V}_k\}) \leq 2$, 
we have 
\begin{equation} \label{eq:xeqv}
	\Span\{X_a\}=\Span\{\widetilde{V}_k\}.
\end{equation} 
Once $X_a$ is computed, an orthonormal basis of 
$\Span\{\widetilde{V}_k\}$ is given by 
\begin{equation}\label{eq:vkdef2}
	V_k = \mbox{orth}(X_a), 
\end{equation}
where $\mbox{orth}(X)$ denotes an orthonormal basis for the range 
of the matrix $X$.
We note that since $\widehat{J}_k$ is of full rank, 
$V_k$ is well-defined (up to an orthogonal transformation) 
even when $A-\mu_kC-\lambda_kI$ is singular.
The approach described here for 
computing a basis of a null space of a matrix via
an augmented system is 
inspired by \cite{Pan2010,Pan2009,Sifuentes2015Randomized} and
can be traced back to \cite{Peters1979}.

After obtaining the orthonormal basis matrix $V_k$ of the desired projection subspace,  we can define the 2DRQ:
\begin{equation} \label{eq:akck}
	(A_k, C_k) \equiv ({V}^H_k A {V}_k, {V}^H_k C {V}_k), 
\end{equation}
where for the sake of exposition, without loss of generality, 
we assume that $V_k$ is up to another orthogonal transformation such that
\begin{equation}\label{eq:determineVk}
	C_k = V_k^H C V_k = \twobytwo{c_{1,k}}{}{}{c_{2,k}}
	\quad \mbox{with} \quad
	c_{1,k} \geq c_{2,k}. 
\end{equation}

When $C_k$ is indefinite, we have the following $2\times 2$ 2DEVP of 
the 2DRQ \eqref{eq:akck}: 
\begin{subequations}\label{projectEVP}
	\begin{empheq}[left={}\empheqlbrace]{alignat=2}
		(A_k-\nu C_k-\theta I)z &= 0, \label{projectEVP1} \\
		z^HC_kz &= 0, \label{projectEVP2} \\
		z^Hz &= 1. \label{projectEVP3} 
	\end{empheq}
\end{subequations}
By Section~3 of Part I \cite{2DEVPI}, we know that for the $2\times 2$ 
2DEVP \eqref{projectEVP}, if $a_{12,k}\neq 0$, 
where $a_{ij,k}$ is the $(i,j)$ element of $A_k$,  
then there are two distinct 2D-eigentriplets of \eqref{projectEVP}
\begin{equation}\label{def:2deigsimple}
	(\nu(\alpha_{k,i}),\theta(\alpha_{k,i}),z(\alpha_{k,i})) 
	\quad \mbox{for} \quad i = 1,2,  
\end{equation} 
where $\alpha_{k,i} = \pm {|a_{12,k}|}/{a_{12,k}}$, and   
\begin{equation*}\label{def:zalpha}
	\nu(\alpha) = \frac{z(\alpha)^HC_kA_kz(\alpha)}{\|C_kz(\alpha)\|^2}, \quad
	\theta(\alpha) = z(\alpha)^HA_kz(\alpha),\quad
	z(\alpha) = \begin{bmatrix}
		\sqrt{\frac{-c_{2,k}}{c_{1,k}-c_{2,k}}}\\
		\alpha\sqrt{\frac{c_{1,k}}{c_{1,k}-c_{2,k}}}
	\end{bmatrix},
\end{equation*}
Otherwise, if $a_{12,k} = 0$, the 2D-eigentriplets of \eqref{projectEVP} 
are given by
\begin{equation} \label{eq:2deigsmultiple} 
	(\nu_1,\theta_1,z(\alpha))\equiv 
	\left(\frac{a_{11,k}-a_{22,k}}{c_{1,k}-c_{2,k}}, 
	\frac{a_{22,k}c_{1,k}-a_{11,k}c_{2,k}}{c_{1,k}-c_{2,k}}, z(\alpha)\right),
\end{equation} 
where $\alpha \in \mathbb{C}$ and $|\alpha|=1$.
From the 2D-eigentriplets \eqref{def:2deigsimple} or \eqref{eq:2deigsmultiple} 
of $(A_k, C_k)$, we can use the following 2D Ritz triplets to define 
the $k+1$st iterate $(\mu_{k+1}, \lambda_{k+1}, x_{k+1})$: 
\begin{equation}   \label{eq:update1a}
	\mu_{k+1}=\nu(\alpha_{k,j}), \quad
	\lambda_{k+1}=\theta(\alpha_{k,j}) \quad \mbox{and} \quad
	x_{k+1} = V_kz(\alpha_{k,j}),
\end{equation}
when $a_{12,k}\neq 0$, where $j$ is the index such that 
$|\mu_k-\nu(\alpha_{k,j})|+|\lambda_k-\theta(\alpha_{k,j})|$ is smaller
one for $j = 1, 2$. Otherwise, when $a_{12,k} = 0$,  
the $k+1$st iterate $(\mu_{k+1}, \lambda_{k+1}, x_{k+1})$ is given by
\begin{equation}   \label{eq:update1b}
	\mu_{k+1}= \nu_1, \quad
	\lambda_{k+1}= \theta_1 
	\quad \mbox{and} \quad
	x_{k+1} = V_kz(1), 
\end{equation}
where for the sake of convenience, we choose $\alpha=1$.

When $C_k$ is not indefinite, as we may encounter at 
early stages of iterations, we propose the following strategy for determining the $k+1$st iterate $(\mu_{k+1}, \lambda_{k+1}, x_{k+1})$. First, since the exact 2D-eigenvector $x_*$ satisfies $x_*^HCx_*=0$, 
we choose a unit vector $x_{k+1}$ to 
minimize $|x^HCx|$ for $x \in \Span\{V_k\}$.
Specifically, when $c_{1,k}\neq c_{2,k}$,  up to a scaling, 
$x_{k+1}$ is uniquely determined by
\begin{equation}\label{eq:updatedefinite1}
	x_{k+1} = \left\{\begin{aligned}
		V_ke_1, & \quad |c_{1,k}|< |c_{2,k}|,\\
		V_ke_2, & \quad |c_{1,k}|> |c_{2,k}|. \\	
	\end{aligned}
	\right.
\end{equation}  
When $c_{1,k} = c_{2,k}$, we use 
\begin{equation}\label{eq:updatedefinite2}
	x_{k+1}=V_kw/\|V_kw\|,
\end{equation}
where $w$ is a uniformly distributed random vector on $[-1,1]$. 
Once $x_{k+1}$ is determined by \eqref{eq:updatedefinite1}
or \eqref{eq:updatedefinite2}, 
$(\mu_{k+1},\lambda_{k+1})$ is obtained
by solving the following least squares problem:
\begin{equation}  \label{eq:leastsquare}
	(\mu_{k+1},\lambda_{k+1}) = 
	\arg\min\limits_{\nu,\theta\in\mathbb{R}} 
	\left\|Ax_{k+1}-\nu Cx_{k+1}-\theta x_{k+1}\right\|.
\end{equation} 

\subsection{Algorithm outline}\label{subsec:outline}
Algorithm~\ref{alg:2dRQIps} summarizes the derivation in
the previous section for an algorithm to compute a 2D-eigentriplet. 
It is called {\em 2DRQI} since the algorithm is an extension of the 
RQI for a Hermitian matrix $A$. 
By~\eqref{2DRQIaugmen} and \eqref{eq:vkdef2}, we see that  
when $A-\mu_{k}C-\lambda_kI$ is nonsingular, 
\[ 
\Span \{V_{k}\} = 
\Span\{(A-\mu_{k}C-\lambda_kI)^{-1}x_k,(A-\mu_{k}C-\lambda_kI)^{-1}Cx_k\}. 
\] 
If $C = 0$ and $\lambda_k$ is the Rayleigh quotient of $A$ and $x_k$, 
then $\Span\{V_{k}\} = \Span\{(A-\lambda_kI)^{-1}x_k\}$ is
the one used 
in the classical RQI, see e.g.~\cite[Section~4.6]{Parlett1987Symmetric}. 
A few remarks of Algorithm~\ref{alg:2dRQIps} are in order. 

\begin{algorithm}[htbp] 
	\caption{2DRQI} \label{alg:2dRQIps}
	\begin{algorithmic}[1]
		\REQUIRE{$n\times n$ Hermitian matrices $A$ and $C$, and $C$ is
			indefinite; initial $(\mu_0,\lambda_0,x_0)$, 
			{\tt tol}, {\tt maxit}.
		}
		
		\ENSURE{An
			approximate 2D-eigentriplet 
			$(\widehat{\mu}, \widehat{\lambda}, \widehat{x})$ 
			and an estimated backward error $\eta_1$.
		}
		
		\FOR{$k=0,1,2,\ldots, \mbox{\tt maxit}$}
		
		\STATE solve the linear system \eqref{2DRQIaugmen} for 
		the $n\times 2$ matrix $X_a$;
		
		\STATE\label{step:Vk} 
		set $V_k = \mbox{orth}(X_a)$ and 
		update $V_k$ to satisfy \eqref{eq:determineVk};
		
		\STATE 
		solve the $2\times2$ 2DEVP~\eqref{projectEVP} of
		$(A_k, C_k) = (V_k^HAV_k, \mbox{Diag}(c_{1,k},c_{2,k}))$;  
		
		\IF{${C}_k$ is indefinite} 
		
		\STATE determine $(\mu_{k+1},\lambda_{k+1},x_{k+1})$
		by~\eqref{eq:update1a} or~\eqref{eq:update1b}; 
		
		\ELSE{}
		\IF{$|c_{1,k}|\neq|c_{2,k}|$}   
		\STATE determine $x_{k+1}$ by~\eqref{eq:updatedefinite1};   
		\ELSE
		\STATE determine $x_{k+1}$ by~\eqref{eq:updatedefinite2}; 
		\ENDIF
		\STATE determine $(\mu_{k+1},\lambda_{k+1})$ 
		by solving \eqref{eq:leastsquare};
		\ENDIF  \label{step:normalization}
		
		\STATE  \label{step:stop} 
		exit for-loop if 
		{$\eta_1(\mu_{k+1},\lambda_{k+1},x_{k+1})\leq {\tt tol}$}
		
		\ENDFOR
		\RETURN 
		$(\widehat{\mu}, \widehat{\lambda}, \widehat{x}) = 
		(\mu_{k+1},\lambda_{k+1},x_{k+1})$ and 
		$\eta_1(\widehat{\mu}, \widehat{\lambda}, \widehat{x})$.
	\end{algorithmic}
\end{algorithm}

\begin{enumerate}
	\item A proper initial $(\mu_0, \lambda_0, x_0)$ 
	is critical for the rapid convergence of the algorithm.
	The initial pair $(\mu_0, \lambda_0)$ should be close 
	to a 2D-eigenvalue of interest. 
	For the initial vector $x_0$,
	we first compute a 2D-Ritz triplet $(\nu, \theta, z)$
	of 2DRQ $(X^H A X, X^H C X)$, 
	where $X$ consists of the two orthonormal eigenvectors 
	corresponding to two eigenvalues of $A-\mu_0C$ closest to $\lambda_0$,
	and then set $x_0$ to be the 2D-Ritz vector $X z$ associated with 
	the 2D-Ritz value $(\nu, \theta)$ closest to $(\mu_0,\lambda_0)$.
	\item To solve the linear system \eqref{2DRQIaugmen}, 
	we should exploit the structure and sparsity of matrices $A$ and $C$.
	See numerical examples in Section~\ref{sec:Experiments}.  
	\item We use an estimate $\eta_1$ of the backward error of approximate 2D-eigentriplet	$(\mu_{k},\lambda_{k},x_{k})$ as the stopping criterion, see Theorem~\ref{Thm:back} in Section~\ref{sec:backerr}. 
	
\end{enumerate}

In Section~\ref{sec:Experiments}, we will provide 
examples to demonstrate that the 2DRQI is locally quadratically convergent. A formal convergence analysis of the 2DRQI will be presented in Part III of this paper \cite{2DEVPIII}. 


\section{Backward error analysis of 2DEVP} \label{sec:backerr}

It is well-known that the backward error of an approximate solution is 
a reliable and effective stopping criterion in an iterative algorithm. 
In this section, we provide a backward error analysis for
an approximate 2D-eigentriplet of the 2DEVP \eqref{2deig}. 
The resulting backward error estimate can be
used as the stopping criterion of the 2DRQI (Algorithm~\ref{alg:2dRQIps}). 
In Section~\ref{sec:dti}, the notion of the backward error analysis 
of the 2DEVP will be extended to applications for
the computation of the distance to instability 
in Section~\ref{sec:dti}. 
We start with the following theorem. 

\begin{theorem} \label{thm:pert2devp}  
	Let $(\widehat{\mu}, \widehat{\lambda}, \widehat{x})$ be an 
	approximate 2D-eigentriplet of $(A,C)$ with 
	$\widehat{\mu}, \widehat{\lambda}\in\mathbb{R}$ and $\|\widehat{x}\|=1$. 
	Then there exist Hermitian matrices $\delta A$ and $\delta C$ such that 
\begin{itemize} 
\item[(i)] $C+\delta C$ is indefinite, and 
\item[(ii)] $(\widehat{\mu},\widehat{\lambda},\widehat{x})$
is an exact 2D-eigentriplet of 
the perturbed matrix pair $(A+\delta A,C+\delta C)$:
\begin{subnumcases}{\label{2deig-perturb}}
\left( A + \delta A - \widehat{\mu} (C + \delta C) \right)\widehat{x} = \widehat{\lambda} \widehat{x},\label{eq:1ap}\\
\hspace{1.64cm}\widehat{x}^H (C+\delta C) \widehat{x}  = 0,\label{eq:1bp}\\
\hspace{3.08cm}\widehat{x}^H\widehat{x}  = 1. \label{eq:1cp}
\end{subnumcases}
\end{itemize} 
\end{theorem} 
\begin{proof}  We prove by construction. 
	We first find the desired perturbation
	matrix $\delta C$ to satisfy \eqref{eq:1bp}. 
	Define $\delta\widehat{C}=-(\widehat{x}^HC\widehat{x})I$. 
	Then 
	\begin{equation}  \label{eq:cdeltac} 
		\widehat{x}^H(C+\delta\widehat{C})\widehat{x}=0. 
	\end{equation}
	If $C+\delta\widehat{C}$ is indefinite, then 
	\eqref{eq:1bp} holds by taking $\delta C = \delta\widehat{C}$. 
	If $C+\delta\widehat{C}$ is not indefinite, 
	then $C+\delta\widehat{C}$ is positive or negative semi-definite. 
	Equation~\eqref{eq:cdeltac} implies $(C+\delta\widehat{C})\widehat{x}=0$. 
	Let $Q$ be an orthogonal matrix with $Qe_1 = \widehat{x}$. 
	Then we have $Q^H(C+\delta\widehat{C})Qe_1=0$ and 
	\[
	Q^H(C+\delta\widehat{C})Q
	= \begin{bmatrix} 
		0 & 0 \\ 
		0 & \widehat{C}_1 
	\end{bmatrix},
	\] 
	%
	%
	where $\widehat{C}_1$ is an $(n-1)$-by-$(n-1)$ matrix. 
	Define 
	\begin{equation}
		\delta C =\delta\widehat{C}+Q\bbordermatrix{
			&1 & 1 & n-2 \cr
			1 &0 & \Delta & 0 \cr
			1 &\Delta & 0 & 0 \cr
			n-2 &0 & 0 & 0}Q^H
	\end{equation}
	with a nonzero scalar $\Delta$. 
	Then it can be verified that 
	$Q^H(C+\delta C)Q$ (and thus $C+\delta C$) is indefinite 
	and \eqref{eq:1bp} holds. 
	
	For finding the desired perturbation matrix $\delta A$ to 
	satisfy \eqref{eq:1ap}, 
	let $\delta\widehat{A}$ be a Hermitian matrix 
	such that $\delta\widehat{A}\, \widehat{x}={h}/\|{h}\|$, 
	${h}=-(A-\widehat{\lambda}I)\widehat{x}+\widehat{\mu}(C+\delta C)\widehat{x}$.
	For example, $\delta\widehat{A}$ can be a 
	Householder matrix \cite[Theorem~2.1.13]{horn2012}. 
	Then it is straightforward to verify that \eqref{eq:1ap} holds with 
	$\delta A=\|h\|\delta\widehat{A}$. This completes the proof. 
\end{proof}


By Theorem~\ref{thm:pert2devp},   
the backward error $\eta$ of an approximate 2D-eigentriplet 
$(\widehat{\mu}, \widehat{\lambda}, \widehat{x})$ of $(A, C)$
is defined as the infimum of normwise relative perturbation 
of $A$ and $C$ such that $(\widehat{\mu},\widehat{\lambda},\widehat{x})$ 
is an exact 2D-eigentriplet 
of the perturbed 2DEVP~\eqref{2deig-perturb}:
\begin{equation}\label{def:backerr}
	\eta \equiv \inf \left\{\epsilon  \ \Big| \
	\exists\, \delta A, \delta C 
	\,\mbox{s.t.} \,
	\|\delta A\|\leq \epsilon\|A\|, \|\delta C\|\leq \epsilon\|C\|, \, 
	\mbox{and \eqref{2deig-perturb} holds}   
	\right\}. 
\end{equation}
The following theorem provides a tight computable estimate of $\eta$.

\begin{theorem}\label{Thm:back}
	Let $(\widehat{\mu}, \widehat{\lambda}, \widehat{x})$ be 
	an approximate 2D-eigentriplet of $(A,C)$ with $\widehat{\mu}, \widehat{\lambda}\in\mathbb{R}$ and $\|\widehat{x}\|=1$, and 
	\begin{equation} \label{eq:backerr1}
		\eta_1 = \max\left\{
		\frac{|\gamma_A|}{\|A\|},\,
		\frac{|\gamma_C|}{\|C\|},\,
		\frac{\|r\|}{\|A\|+|\widehat{\mu}|\|C\|}
		\right\}, 
	\end{equation}
	where 
	$\gamma_A=\widehat{x}^HA\widehat{x}-\widehat{\lambda}$,
	$\gamma_C=\widehat{x}^HC\widehat{x}$ 
	and $r= (A-\widehat{\mu} C-\widehat{\lambda} I)\widehat{x}$. 
	Then the backward error $\eta$ of 
	$(\widehat{\mu}, \widehat{\lambda}, \widehat{x})$ 
	defined in \eqref{def:backerr} satisfies
	\begin{equation} \label{eq:backerrbd}
		\eta_1\leq\eta\leq\sqrt{2}\,\eta_1.
	\end{equation}
\end{theorem}
\begin{proof} We first prove the lower bound $\eta\geq\eta_1$. 
	For any $\delta A$ and $\delta C$ satisfying the perturbed 
	2DEVP~\eqref{2deig-perturb}, by \eqref{eq:1ap} and \eqref{eq:1bp}, 
	we have $\widehat{x}^H(A+\delta A)\widehat{x}= \widehat\lambda$. 
	Hence by the definition~\eqref{def:backerr} of $\eta$, we have
	\begin{equation} \label{eq:gammaC} 
		\eta \ge \frac{\|\delta C\|}{\|C\|} 
		\ge \frac{|\widehat{x}^H\delta C\widehat{x}|}{\|C\|} 
		= \frac{|\widehat{x}^HC\widehat{x}|}{\|C\|} 
		= \frac{|\gamma_C|}{\|C\|}, 
	\end{equation} 
	and 
	\begin{equation} \label{eq:gammaA} 
		\eta \ge \frac{\|\delta A\|}{\|A\|} \ge
		\dfrac{|\widehat{x}^H\delta A \widehat{x}|}{\|A\|} 
		= \dfrac{|\widehat{x}^HA\widehat{x}-\widehat\lambda |}{\|A\|}
		= \dfrac{|\gamma_A|}{\|A\|}.
	\end{equation} 
	Now by \eqref{eq:1ap},  for the norm of the residual vector $r$: 
	\begin{align*} 
		\|r\| 
		= \| (A-\widehat\mu C -\widehat\lambda I) \widehat{x} \|  
		= \| (\delta A -\widehat\mu \delta C) \widehat{x} \|  
		\le \|\delta A\|+|\widehat{\mu}| \|\delta C\| 
		\le (\|A\|+|\widehat{\mu}| \|C\|) \epsilon.
	\end{align*} 
	Therefore, by the definition~\eqref{def:backerr} of $\eta$, we have
	\begin{equation} \label{eq:rsbd} 
		\eta \ge \frac{\|r\|}{\|A\|+|\widehat{\mu}|\|C\|}.
	\end{equation} 
	Combining \eqref{eq:gammaC}, \eqref{eq:gammaA} and \eqref{eq:rsbd}, 
	we have $\eta \ge \eta_1$.
	
	The gist of finding the upper bound of $\eta$, 
	namely $\eta \leq \sqrt{2}\eta_1$, 
	is to find two particular perturbation matrices $\delta{A}$ and $\delta{C}$ 
	such that 
	\begin{subnumcases}{\label{eq:constructDelta}}
		\delta{A}\widehat{x}-\widehat{\mu}\,\delta{C}\widehat{x} =-r, \label{eq:constructDelta1} \\
		\hspace{2.5em}\widehat{x}^H\delta{C}\widehat{x} =- \gamma_C, \label{eq:constructDelta2}	
	\end{subnumcases}
	and 
	\begin{equation}\label{eq:deltaCReq}
		\mbox{$C+\delta C$ is indefinite},   
	\end{equation} 
	and then derive the upper bound of $\eta$ from the upper bound of 
	$\max\left\{\frac{\|\delta A\|}{\|A\|},\frac{\|\delta C\|}{\|C\|}\right\}$. 
	We first note that we can safely discard the condition \eqref{eq:deltaCReq}. 
	This is due the fact that 
	when \eqref{eq:constructDelta} holds, using the same 
	arguments as in the proof of Theorem~\ref{thm:pert2devp}, 
	we can add infinitesimal perturbation to $\delta A, \delta C$ 
	to guarantee \eqref{eq:constructDelta} and \eqref{eq:deltaCReq} hold. 
	Since the backward error $\eta$ takes the infimum, 
	the quantity 
	$\max\left\{\frac{\|\delta A\|}{\|A\|},\frac{\|\delta C\|}{\|C\|}\right\}$ 
	is still an upper bound.  
	
	To find $\delta A$ and $\delta C$ satisfying
	\eqref{eq:constructDelta}, let us define
	\begin{equation*}\label{eq:tildeac}
		\widetilde{a} \equiv
		-\dfrac{\|A\|}{\|A\|+|\widehat\mu|\,\|C\|} (I-\widehat{x}\widehat{x}^H)r, \quad 
		\widetilde{c} \equiv 
		\dfrac{\sign (\widehat\mu)\|C\|}{\|A\|+|\widehat\mu|\,\|C\|}
		(I-\widehat{x}\widehat{x}^H)r,
	\end{equation*}
	where $\sign (\widehat\mu)=\widehat\mu/|\widehat\mu|$. 
	Then $\widetilde{a}$ and $\widetilde{c}$ are orthogonal to $\widehat{x}$ and
	satisfy
	\begin{equation*} 
		\widetilde{a}-\widehat{\mu}\,\widetilde{c} 
		= -(I-\widehat{x}\widehat{x}^H)\, r.
	\end{equation*}
	Next, let us define 
	$a \equiv (-\widehat{x}^Hr-\widehat{\mu}\gamma_C)\widehat{x}+\widetilde{a}$
	and $c \equiv -\gamma_C\widehat{x}+\widetilde{c}$.
	Then it holds that
	\[
	\left\{
	\begin{aligned}
		a - \widehat{\mu}c &= (\widetilde{a}-\widehat{\mu}\widetilde{c})-\widehat{x}^Hr\widehat{x}= -(I-\widehat{x}\widehat{x}^H)r-\widehat{x}^Hr\widehat{x} = -r,\\
		\widehat{x}^Hc &= -\gamma_C.
	\end{aligned}
	\right.
	\] 
	From the vectors $a$ and $c$, we 
	can construct Hermitian matrices $\delta A$ and $\delta C$, 
	say real constant multiples of Householder 
	reflections \cite[Theorem~2.1.13]{horn2012} satisfying
	$\delta A \widehat{x} = a$, 
	$\delta C\widehat{x} = c$
	and $\|\delta A\| = \|a\|$ and $\|\delta C\| = \|c\|$.
	Then $\delta A$ and $\delta C$ are desired matrices 
	satisfying \eqref{eq:constructDelta}. 
	
	For $\delta A$, by the definition of $r$, we have $-\widehat{x}^Hr-\widehat{\mu}\gamma_C = -\gamma_A$, and thus
	\begin{align}
		\dfrac{\|\delta A\|}{\|A\|} &= \dfrac{\|a\|}{\|A\|}
		= \dfrac{\|-\gamma_A\widehat{x}+\widetilde{a}\|}{\|A\|}
		= \dfrac{\sqrt{\gamma_A^2+\|\widetilde{a}\|^2}}{\|A\|} 
		= \sqrt{\left(\dfrac{|\gamma_A|}{\|A\|}\right)^2+
			\left(\dfrac{\|\widetilde{a}\|}{\|A\|}\right)^2} \nonumber \\
		&\le \sqrt{\left(\dfrac{|\gamma_A|}{\|A\|}\right)^2+
			\left(\dfrac{\|r\|}{\|A\|+|\widehat{\mu}|\, \|C\|}\right)^2} 
		\le \sqrt{2}\eta_1. \label{eq:deltaAbd0} 
	\end{align}
	By an analogous derivation, for $\delta C$, we have 
	\begin{equation} \label{eq:deltaCbd} 
		\dfrac{\|\delta C\|}{\|C\|}\le \sqrt{2}\eta_1.
	\end{equation} 
	{Combining the upper bounds \eqref{eq:deltaAbd0} and 
		\eqref{eq:deltaCbd}, we have
		$\eta \le \sqrt{2} \eta_1$.} 
	This completes the proof. 
\end{proof}

\section{Applications}\label{sec:apps}

\subsection{Minmax of Rayleigh Quotients}\label{sec:minmax2RQs}

In Sections 2.1 and 7.1 of Part I \cite{2DEVPI}, 
we discussed an application of the 2DEVP for
the minmax of Rayleigh quotients
(RQminmax) of $n\times n$ Hermitian matrices $A$ and $B$: 
\begin{equation}\label{prob:minmaxRQ}
\min_{x\neq 0}\max \left\{
\frac{x^HAx}{x^Hx},\, 
\frac{x^HBx}{x^Hx}\right\}
\end{equation}
In \cite[Theorem~2.1]{2DEVPI},  
we have shown that the RQminmax~\eqref{prob:minmaxRQ}
can be divided into three cases. For Case-I and Case-II,
one need to calculate the eigen-subspace corresponding to 
the minimum eigenvalues $\lambda_A$ and $\lambda_B$ of $A$ and $B$. 
This could be computational expensive when 
the multiplicity of $\lambda_A$ or $\lambda_B$ is larger than 1.
In the following we derive a computational-friendly variant 
of \cite[Theorem~2.1]{2DEVPI}, which only need to calculate 
an eigenvector corresponding to $\lambda_A$ and $\lambda_B$,  
regardless of the multiplicities of $\lambda_A$ and $\lambda_B$.

\begin{theorem}\label{thm:classfication1}  
Let $\lambda_A$ be the minimum eigenvalue 
and $x_A$ be a correponding eigenvector of $A$, 
$\lambda_B$ be the minimum eigenvalue 
and $x_B$ be a correponding eigenvector of $B$,
$\rho_A(x) = x^HA x/x^H x$ and  $\rho_B(x) = x^HB x/x^H x$ 
be the Rayleigh quotients of $A$ and $B$, respectively. 
\begin{enumerate}[I.]
\item\label{item:B1}
If $\lambda_A\geq \rho_B(x_A)$, then $x_A$ is a solution 
of the RQminmax \eqref{prob:minmaxRQ};
\item\label{item:B2}
If $\lambda_B\geq \rho_A(x_B)$, 
then $x_B$ is a solution of the RQminmax \eqref{prob:minmaxRQ};
\item\label{item:B3}
Otherwise, namely $\lambda_A < \rho_B(x_A)$ and $\lambda_B < \rho_A(x_B)$, 
let $\mu_*$ be an optimizer of the eigenvalue optimization 
problem (EVopt): 
\begin{equation}\label{eq:EVopt0}
\max\limits_{\mu\in \mathbb{R}}\lambda_{\min}(A-\mu C),
\end{equation}
and $V_{\mu_*}$ be the set of eigenvectors $x_*$ corresponding to 
$\lambda_{\min}(A-\mu_*C)$ and $x_*^HCx_*=0$, where $C = A - B$, 
then 
(a) $\mu_*\in[0,1]$, 
(b) $V_{\mu_*}\neq\emptyset$ and 
(c) any $x_*\in V_{\mu_*}$ is a solution of the RQminmax~\eqref{prob:minmaxRQ}.
\end{enumerate}
\end{theorem}

\begin{proof}
For Case-\ref{item:B1}, we note that for any $x\neq0$, 
\[
\max\{\rho_A(x),\rho_B(x)\} \geq \rho_A(x)\geq\lambda_A.
\]
On the other hand,  
\[
\max\{\rho_A(x_A),\rho_B(x_A)\}
= \lambda_A.
\] 
Thus $x_*=x_A$ is a solution of the RQminmax~\eqref{prob:minmaxRQ}. 
%

Case-II can be proven by exchanging the roles of 
$A$ and $B$ in the proof of Case-I.
 
For Case-\ref{item:B3}, we need to prove that under the conditions 
$\lambda_A < \rho_B(x_A)$ and $\lambda_B < \rho_A(x_B)$, 
we have the results (a), (b) and (c). 
To that end, let $S_A$ and $S_B$ be orthonormal bases of the eigensubspaces of $\lambda_A$ and $\lambda_B$, respectively, and denote $\theta_A=\lambda_{\min}(S^H_B A S_B)$, $\theta_B = \lambda_{\min}(S^H_A B S_A)$. Let us divide Case-\ref{item:B3} into 
subcases based on the relation between 
$\lambda_A$ ($\lambda_B$) and $\theta_B$ ($\theta_B$). 
If $\lambda_A < \theta_B$ 
and $\lambda_B < \theta_A$, 
then it belongs to the case of Theorem 2.1(III) of Part I \cite{2DEVPI}, and 
$$
\arg\min\limits_{x\neq0}\Big( \max\{\rho_A(x),\rho_B(x)\}\Big) = V_{\mu_*}. 
$$
The results (b) and (c) are immediately followed. 
The result (a) is contained in \cite[Theorem~7.1]{2DEVPI}, namely
\[
\arg\max\limits_{\mu\in \mathbb{R}}g(\mu)
= \arg\max\limits_{\mu\in (0,1)} g(\mu), 
\]
where $g(\mu) = \lambda_{\min}(A-\mu C)$. 
Consequently, we only need to consider 
the subcase where the inequalities 
$\lambda_A < \theta_B$ and $\lambda_B < \theta_A$
do not hold simultaneously.
This implies that at least one of the following conditions holds: 
(i) $\lambda_A \geq \theta_B$;  
(ii) $\lambda_B \geq \theta_A$. 
Let us consider (i) in the following.  (ii) can be shown analogously. 

By the condition under Case~\ref{item:B3}, i.e.,
$\lambda_A<\rho_B(x_A)$ and $\lambda_B<\rho_A(x_B)$, 
we have
\begin{equation}\label{eq:caseIIIcond}
-x_A^HCx_A=-\lambda_A+\rho_B(x_A)>0, \quad 
-x_B^HCx_B=-\rho_A(x_B)+\lambda_B<0. 
\end{equation}
Note that $x_A$ and $x_B$ are also eigenvectors of $A-0\cdot C=A$ 
and $A-1\cdot C=B$, respectively. Thus 
by \cite[Theorem~4.5]{2DEVPI}, the inequalities in \eqref{eq:caseIIIcond} imply that 
\begin{equation}\label{eq:dq}
g^{'(-)}(0)=\lambda_{\max}(-S_A^HCS_A)\geq -x_A^HCx_A>0, \quad g^{'(+)}(1)=\lambda_{\min}(-S_B^HCS_B)\leq -x_B^HCx_B<0.
\end{equation}
Let $\mu_*$ be an optimizer of EVopt~\eqref{eq:EVopt0}, 
then by \eqref{eq:dq} and the concavity of 
$g(\mu)$, we conclude that $\mu_*\in[0,1]$.  
This completes the proof of the result (a).

For result (b), note that $(\mu_*,\lambda_{\min}(A-\mu_*C))$ 
is a 2D eigenvalue of $(A,C)$ according to 
\cite[Theorem~5.1]{2DEVPI}. Then the associated 2D-eigenvectors 
belong to $V_{\mu_*}$ and thus we obtain the result (b).

To prove the result (c), we first calculate the optimal value of
the EVopt~\eqref{eq:EVopt0}. 
Since $\lambda_A \geq \theta_B$, by denoting $z_B$ as the 
eigenvector of $S_A^HBS_A$ corresponding to $\theta_B$ and 
the definition of $S_A$, we have
$$
\rho_A(S_Az_B) = \lambda_A \geq \theta_B = \rho_B(S_Az_B).
$$ 
Let $\tilde{x}_A = S_Az_B$. Then 
$-\widetilde{x}_A^HC\widetilde{x}_A\leq0$ and thus 
by \cite[Theorem~4.5]{2DEVPI},
\begin{equation}\label{eq:dq2}
g_{+}'(0)=\lambda_{min}(-S_A^HCS_A)\leq -\widetilde{x}_A^HC\widetilde{x}_A\leq0.
\end{equation}
According to \eqref{eq:dq}, \eqref{eq:dq2} and the concavity of 
$g(\mu)=\lambda_{\min}(A-\mu C)$ (see \cite[Theorem~4.1]{2DEVPI}), 
$0$ is an optimizer of EVopt~\eqref{eq:EVopt0} 
and thus
\begin{equation}\label{eq:gmax}
\max\limits_{\mu\in\mathbb{R}}\lambda_{\min}(A-\mu C) =\lambda_A.
\end{equation} 
Now for any $x_*\in V_{\mu_*}$, we have 
\[
\rho_B(x_*) = \rho_A(x_*) = \lambda_{\min}(A-\mu_*C) 
= \max\limits_{\mu\in\mathbb{R}}\lambda_{\min}(A-\mu C) 
= \lambda_A = \min\max\{\rho_A(x),\rho_{B}(x)\},
\]
where the first equality results from  $x_*^HCx_*=0$, 
the second equality results from the fact that $\rho_A(x_*) = \rho_{A-\mu_*C}(x_*)$ and $x_*$ is an eigenvector corresponding to $\lambda_{\min}(A-\mu_*C)$, 
the third equality comes from $\mu_*$ is an optimizer,  
the fourth equality results from \eqref{eq:gmax}
and the last equality holds according to 
Theorem 2.1(I) of Part I \cite{2DEVPI} 
as $\theta_B\leq\lambda_A$. 
Thus $x_*$ is the solution to the RQminmax~\eqref{prob:minmaxRQ}. 
This completes the proof of the result (c).
\end{proof}

By Defintion~5.1 of Part I \cite{2DEVPI}, 
for Case~\ref{item:B3} of Theorem~\ref{thm:classfication1}, 
we know that $(\mu_*,\lambda_*)$ with $\lambda_*=\lambda_{\min}(A-\mu_*C)$ 
is the minimum 2D-eigenvalue of $(A,C)$. 
On the other hand, by the definition of $V_{\mu_*}$, up to a scaling, 
$x_*\in V_{\mu_*}$ if and only if $x_*$ is a 2D-eigenvector associated 
with $(\mu_*,\lambda_*)$. Thus the RQminmax~\eqref{prob:minmaxRQ}, 
excluding Cases~\ref{item:B1} and \ref{item:B2} 
in Theorem~\ref{thm:classfication1}, turns to 
calculating a minimum 2D-eigenvalue 
and the corresponding 2D-eigenvector of $(A, C)$.
 
Based on the fact that
$\mu_*$ of the minimum 2D-eigenvalue $(\mu_*,\lambda_*)$
must be in $[0,1]$, we can derive a combination of the bisection search and 
the 2DRQI (Algorithm~\ref{alg:2dRQIps}).
Starting with the search interval $[a,b] = [0, 1]$
of the EVopt~\eqref{eq:EVopt0}, let 
\begin{equation} \label{eq:initials}
\mu_0=\frac{1}{2}(a+b) 
\quad \mbox{and} \quad
\lambda_0= \lambda_{\min}(A-\mu_0C)
\end{equation} 
and $x_0$ be the one as recommended
for the 2DRQI (Algorithm~\ref{alg:2dRQIps}). 
Then we can use the 2DRQI with the initial 
$(\mu_0, \lambda_0, x_0)$ to 
find a 2D-eigentriplet $(\widehat{\mu},\widehat{\lambda}, \widehat{x})$ 
of $(A,C)$. 

If $\widehat{\lambda} = \lambda_{\min}(A - \widehat{\mu} C)$,  
then according to \cite[Corollary 5.1]{2DEVPI}, \[
\widehat{\lambda} = 
\max\limits_{\mu\in \mathbb{R}}\lambda_{\min}(A-\mu C).
\]
Thus $\widehat{\mu}$ is an optimizer of EVopt \eqref{eq:EVopt0} 
and $\widehat{x}$ is the solution of RQminmax~\eqref{prob:minmaxRQ}.

If $\widehat{\lambda} \neq  \lambda_{\min}(A - \widehat{\mu} C)$, 
or the 2DRQI does not converge, then we can use the concavity 
of $g(\mu) = \lambda_{\min}(A-\mu C)$ (see \cite[Theorem~4.1]{2DEVPI})
to bisect the interval $[a,b]$ and run the 2DRQI with a new initial 
$(\mu_0, \lambda_0, x_0)$. 
This bisection seach strategy works due to the facts that 
\begin{itemize} 
	\item 
	if $(x^{(n)})^HCx^{(n)}\leq0$, where $x^{(n)}$ is 
	an eigenvector corresponding to 
	$\lambda_0 = \lambda_{\min}(A - \mu_0 C)$, then 
	by \cite[Theorem~4.5]{2DEVPI}, 
	$g_{-}'(\mu_0)=\lambda_{\max}(-X_0(\mu_0)^HCX_0(\mu_0))\geq0$, 
	where $X_0(\mu)$ is an orthonormal 
	basis of the eigen-subspace of $\lambda_{\min}(A-\mu C)$, 
	and there is an optimizer $\mu_*$ of the EVopt~\eqref{eq:EVopt0} 
	such that $\mu_*\geq\mu_0$. Consequently, we set $a = \mu_0$ to 
	half the search interval. 
	
	\item 
	if $(x^{(n)})^HCx^{(n)}>0$, then by \cite[Theorem~4.5]{2DEVPI},
	$g_{+}'(\mu_0)=\lambda_{\min}(-X_0(\mu_0)^HCX_0(\mu_0))<0$ and 
	there is an optimizer $\mu_*$of the EVopt~\eqref{eq:EVopt0} such that 
	$\mu_*\leq\mu_0$. Consequently, we set $b = \mu_0$ to half 
	the search interval.
\end{itemize} 

A combination of the 2DRQI (Algorithm~\ref{alg:2dRQIps}) 
and the bisection search described above 
is summarized in Algorithm~\ref{alg:minmax2RQs}
for solving the RQminmax~\eqref{prob:minmaxRQ}, 
where in line 9 we use whether
\[
|\widehat{\lambda}-\lambda_{\min}(A-\widehat{\mu}C)|
=
|\widehat{\lambda}-\lambda_{\min}\left((1-\widehat{\mu})A+\widehat{\mu}B\right)|
< {\tt reltol}\cdot (|1-\widehat{\mu}|\|A\| + |\widehat{\mu}|\|B\|)
\] 
to numerically check whether 
$\widehat{\lambda}=\lambda_{\min}(A-\widehat{\mu}C)$.

\begin{algorithm} 
	\caption{Minimax of two RQs} \label{alg:minmax2RQs}
	\begin{algorithmic}[1]
		\REQUIRE{
			$n$-by-$n$ Hermitian matrices $A$ and $B$, 
			tolerance values {\tt reltol} and {\tt backtol}.
		} 
		
		\ENSURE{
			approximate solution $\widehat{x}$ and the optimal value $\widehat{\lambda}$ 
			of RQminmax \eqref{prob:minmaxRQ}}
		
		\STATE\label{step:A1} 
		compute a minimum eigenpair $(\lambda_A,x_A)$ of $A$.
		If $\lambda_A\geq\rho_{B}(x_A)$, then return
		$(\widehat{\lambda}, \widehat{x}) =(\lambda_A,x_A)$.
		
		\STATE \label{step:A2} 
		compute a minimum eigenpair $(\lambda_B,x_B)$ of $B$.
		If $\lambda_B\geq\rho_{A}(x_B)$, then return
		$(\widehat{\lambda}, \widehat{x}) =(\lambda_B,x_B)$.
		
		\STATE 
		set $[a,b] = [0,1]$;
		
		\FOR{$k=0,1,2,\ldots,$} 
		
		\STATE \label{step:mu0}  
		set $\mu_0={(a+b)}/{2}$;
		
		\STATE  \label{step:lambda0} 
		compute two smallest eigenpairs 
		$(\lambda_n, x^{(n)})$,  
		$(\lambda_{n-1}, x^{(n-1)})$ 
		of $A-\mu_0C$;
		
		\STATE \label{step:10} 
		compute the minmum 2D-Ritz triplet $(\nu, \theta, z)$
		of $(Z^HAZ,Z^HCZ)$, where
		$Z=\begin{bmatrix} x^{(n-1)} & x^{(n)}\end{bmatrix}$;
		
		\STATE\label{line:initial} apply the 2DRQI (Algorithm~\ref{alg:2dRQIps}) 
		with the initial 
		$(\mu_0,\lambda_0 = \lambda_n, x_0 = Z z)$ 
		and the backward error tolerance {\tt backtol}.
		
		\IF{2DRQI converges to 
			$(\widehat{\mu}, \widehat{\lambda},\widehat{x})$ 
			and 
			$|\widehat{\lambda}-\lambda_{\min}(A-\widehat{\mu}C)|<
			{\tt reltol}\cdot {(|1-\widehat{\mu}|\|A\| + |\widehat{\mu}|\|B\|)}$
		}   
		\label{step:check}	
		
		\RETURN\label{step:goodReturn}  
		$(\widehat{\lambda},\widehat{x})$.

		\ELSE 
		
		\IF{$(x^{(n)})^H Cx^{(n)}\leq0$}
		
		\STATE 
		update $a = \mu_0$.
		
		\ELSE
		
		\STATE 
		update $b = \mu_0$.
		
		\ENDIF
		
		\ENDIF
		
		\ENDFOR			
	\end{algorithmic}
\end{algorithm}

For a robust implementation, we need to deal with the extreme case 
where the 2DRQI (Algorithm~\ref{alg:2dRQIps}) does not converge 
to the correct 2D eigentriplet even when $b-a$ is sufficiently small. 
Specifically, we terminate the outer iteration when 
$b-a<{\tt abstol}$, where ${\tt abstol}$ is a prescribed tolerance. 
According to our analysis, $\mu_*\in[a,b]$ and thus 
$\widehat{\mu}=(a+b)/2, \widehat{\lambda}=\lambda_{\min}(A-\widehat{\mu}C)$ 
is already sufficiently close to the optimizer of the EVopt~\eqref{eq:EVopt0}. 
The remaining issue is how to recover an approximation $\widehat{x}_*$ to 
the solution of the RQminmax~\eqref{prob:minmaxRQ}, namely
$\widehat{x}_*^HC\widehat{x}_* \approx 0$.   

To that end, we compute eigenvectors $x_a, x_b$ associated 
with the minimum eigenvalue of $A-aC$ and $A-bC$. After proper scaling 
we can assume $\|x_a\|=\|x_b\|=1$, and $x_a^Hx_b$ is real and non-negative. 
If $x_a^HCx_a=0$ or $x_b^HCx_b=0$ or $x_a\parallel x_b$, i.e., 
$x_a = \alpha x_b$ for a constant $\alpha$, 
we set $\widehat{x}=x_a$ or $\widehat{x}=x_b$ or $\widehat{x}=x_a$. 
Otherwise, we denote $q_a=x_a$, 
$q_b=\frac{(I-x_ax_a^H)x_b}{\|(I-x_ax_a^H)x_b\|}
=\frac{(I-x_ax_a^H)x_b}{\sqrt{1-(x_a^Hx_b)^2}}$ and 
$u(\theta)=\cos\theta q_a+\sin\theta q_b$. We find $\widehat{\theta}$ 
that minimizes $|u(\theta)^HCu(\theta)|$ among $[0, \arccos(x_a^Hx_b)]$. 
We then set $\widehat{x}=u(\widehat{\theta})$ and return 
$(\widehat{\lambda},\widehat{x})$ as the solution 
of the RQminmax~\eqref{prob:minmaxRQ}. 
The following proposition shws that this strategy is valid 
under mild conditions.

\begin{proposition}\label{prop:strategy}
If $\lambda_{\min}(A-aC)$ and $\lambda_{\min}(A-bC)$ are simple eigenvalues, 
then 
\begin{itemize} 
\item[(a)] $\widehat{x}^HA\widehat{x}= \widehat{x}^HB\widehat{x}$.
\item[(b)] $\|(A-\mu_*C)\widehat{x}-\lambda_*\widehat{x}\|\leq 6(b-a)\|C\|$.
\item[(c)] $|\widehat{x}^HA\widehat{x}-\lambda_*|\leq 6(b-a)\|C\|$.
\end{itemize} 
\end{proposition}
\begin{proof}
First consider the case $x_a^HCx_a=0$. 
Let $\widehat{x}=x_a$, then the result (a) holds. The result (c) 
holds due to the fact that 
\[
|\widehat{x}^HA\widehat{x}-\lambda_*|
=|\lambda_{\min}(A-aC)-\lambda_{\min}(A-\mu_*C)|\leq (b-a)\|C\|,
\]
where the last inequality comes 
from Weyl~theorem~\cite[p.\,203, Corollary 4.10]{Stewart1990Sun}. 
The result (b) holds since
\begin{equation}\label{eq:propstrategy1}
\begin{aligned}
(A-\mu_*C)x_a-\lambda_*x_a &= (A-aC)x_a-
\lambda_{\min}(A-aC)x_a - (\mu_*-a)Cx_a - (\lambda_*-\lambda_{\min}(A-aC))x_a\\
	&=- (\mu_*-a)Cx_a - (\lambda_*-\lambda_{\min}(A-aC))x_a
\end{aligned}
\end{equation}
and 
\begin{equation}\label{eq:propstrategy2}
\|(A-\mu_*C)x_a-\lambda_*x_a\| \leq 2(b-a)\|C\|.
\end{equation}

The argument for the case $x_b^HCx_b=0$ is similar. 
The remaining is the case where both $x_a^HCx_a$ and $x_b^HCx_b$ are nonzero.

According to the concavity of $g(\mu)=\lambda_{\min}(A-\mu C)$ and 
$a\leq\mu_*\leq b$, $g'(a)\geq0$ and $g'(b)\leq0$. Thus we have 
$x_a^HCx_a\leq0$ and $x_b^HCx_b\geq0$. Since both $x_a^HCx_a$ and 
$x_b^HCx_b$ are nonzero, we have $x_a^HCx_a<0$ and $x_b^HCx_b>0$. 
This implies $x_a\nparallel x_b$, and thus $q_a,q_b$ are well-defined.

Note that $q_a^Hq_b=0$ and $\|q_a\|=\|q_b\|=1$. Then $\|u(\theta)\|=1$ for all $\theta$. Furthermore, let $\theta_a=0$, $\theta_b=\arccos(x_a^Hx_b)$. Straight calculation shows that $u(\theta_a)=x_a$ and $u(\theta_b) = x_b$. 

Define function $h(\theta) = u(\theta)^HCu(\theta)$. Then 
$h(\theta_a)=x_a^HCx_a<0$ and $h(\theta_b)=x_b^HCx_b>0$. By continuity, 
there exists $\widehat{\theta}\in(\theta_a,\theta_b)$ such that 
$h(\widehat{\theta})=0$. Thus 
$\min\limits_{\theta\in[\theta_a,\theta_b]} |u(\theta)^HCu(\theta)|=0$ 
and we have $\widehat{x}^HC\widehat{x}=0$. The result (a) is obtained. 

We next show $\widehat{x}$ lies approximately in the eigen-subspace of 
$A-\mu_*C$ in the backward sense.  
Denote $r_a,r_b$ such that 
$(A-\mu_*C)x_a = \lambda_*x_a + r_a$, $(A-\mu_*C)x_b = \lambda_*x_b + r_b$. 
We have
\begin{equation}\label{eq:propstrategy3}
	\begin{aligned}
		(A-\mu_*C)\widehat{x}&=(A-\mu_*C)\left(\cos\widehat{\theta}x_a+\sin\widehat{\theta}\frac{(I-x_ax_a^H)x_b}{\sqrt{1-(x_a^Hx_b)^2}}\right)\\
		&=\cos\widehat{\theta}(\lambda_*x_a+r_a)+\sin\widehat{\theta}\frac{\lambda_*x_b+r_b-x_a^Hx_b(\lambda_*x_a+r_a)}{\sqrt{1-(x_a^Hx_b)^2}}\\
		&=\lambda_*\widehat{x}+\cos\widehat{\theta}r_a+\sin\widehat{\theta}\frac{r_b-r_ax_a^Hx_b}{\sqrt{1-(x_a^Hx_b)^2}}\\
		&\equiv\lambda_*\widehat{x}+\widehat{r},
	\end{aligned}
\end{equation}
with
\begin{equation}
	\begin{aligned}
		\|\widehat{r}\|&\leq\|r_a\|+\sin\theta_b\frac{\|r_b-r_ax_a^Hx_b\|}{\sqrt{1-(x_a^Hx_b)^2}}\\
		&=\|r_a\|+\|r_b-r_ax_a^Hx_b\|\\
		&\leq2\|r_a\|+\|r_b\|.
	\end{aligned}
\end{equation}
Note that using the same argument in \eqref{eq:propstrategy1} and \eqref{eq:propstrategy2}, we can obtain $\|r_a\|\leq 2(b-a)\|C\|$ and $\|r_b\|\leq 2(b-a)\|C\|$. Thus we have
\[\|\widehat{r}\|\leq 6(b-a)\|C\|\]
and we reach the result (b). Multiplying $\widehat{x}^H$ on the 
left of \eqref{eq:propstrategy3} leads to the result (c).
\end{proof}

Numerical examples for large scale RQminmax~\eqref{prob:minmaxRQ} arising
from signal processing will be presented in Section~\ref{sec:Experiments}.


\subsection{The distance to instability} \label{sec:dti} 

As discussed in Section 2.2 of Part I \cite{2DEVPI}, 
the distance to instability (DTI) of a stable matrix 
$\widehat{A} \in \mathbb{C}^{n\times n}$ can be recast as
the eigenvalue optimization as follows: 
\begin{equation} \label{eq:47}
\beta(\widehat{A}) 
\equiv \min\left\{\|E\| \mid \widehat{A} + E\ \mbox{is unstable}\right\}  
 = \min\limits_{\mu\in\mathbb{R}} \lambda_n(\mu), 
\end{equation} 
where $\lambda_n(\mu)$ is the smallest positive eigenvalue
of $A - \mu C$ with 
\begin{equation} \label{eq:defA}
A = \begin{bmatrix}
&\widehat{A}\\
\widehat{A}^H&
\end{bmatrix} 
\quad \mbox{and} \quad
C = \begin{bmatrix}
& jI\\
-jI &
\end{bmatrix}. 
\end{equation}
Furthermore, in Section 7.2 of Part I \cite{2DEVPI}, we know that
if $\mu_*$ is an optimizer of \eqref{eq:47}, then 
$(\mu_*, \lambda_n(\mu_*) = \beta(\widehat{A}) )$
is a 2D-eigenvalue of $(A,C)$ and $\mu_*\in\left[-\|A\|,\|A\|\right]$, 
and 
\begin{subequations}  \label{eq:beta2D} 
\begin{align}
\beta(\widehat{A}) 
&=\min\{\lambda \mid 
	\mbox{$(\mu,\lambda)$ is a 2D-eigenvalue of $(A,C)$ and $\lambda>0$}
	\} \label{eq:beta2D:1} \\
&=-\max\{\lambda \mid 
	\mbox{$(\mu,\lambda)$ is a 2D-eigenvalue of $(A,C)$ and $\lambda<0$}
	\} \label{eq:beta2D:2} \\
&=\min\{|\lambda| \mid 
	\mbox{$(\mu,\lambda)$ is a 2D-eigenvalue of $(A,C)$} \}.
	 \label{eq:beta2D:3}
\end{align}
\end{subequations} 
In addition, we note that by the structure of $A$ and 
$C$ in \eqref{eq:defA} and
equations~\eqref{eq:1b} and \eqref{eq:1c} of the 2DEVP \eqref{2deig},
the corresponding 2D-eigenvector  
$x_*=\left[\begin{smallmatrix}
x_1\\
x_2
\end{smallmatrix}\right]$ 
of $(\mu_*, \lambda_n(\mu_*))$ 
must obey 
\begin{equation}\label{eq:xhatcondition0}
\imag(x_1^Hx_2)=0 
\quad \mbox{and} \quad 
x_1^Hx_1 + x_2^Hx_2=1.
\end{equation}
Meanwhile, from equation~\eqref{eq:1a}, 
\[
\begin{aligned}
\widehat{A}x_2
&=\mu_* jx_2 + \beta(\widehat{A}) x_1,\\
\widehat{A}^Hx_1
&=-\mu_* jx_1 + \beta(\widehat{A}) x_2.
\end{aligned}\]
Since $x_1^H\widehat{A}x_2 =\overline{x_2^H\widehat{A}^Hx_1}$, we have
\[
\mu_* jx_1^Hx_2+\beta(\widehat{A}) x_1^Hx_1=\overline{-\mu_* jx_2^Hx_1+\beta(\widehat{A}) x_2^Hx_2},
\]
which, by noting $\beta(\widehat{A})\neq0$, is equivalent to 
\begin{equation}\label{eq:xhatcondition1}
x_1^Hx_1=x_2^Hx_2.
\end{equation} 
Hence the 2D-eigenvector $x_*$ must satisfy 
the relations \eqref{eq:xhatcondition0} and \eqref{eq:xhatcondition1}.

\begin{algorithm}[htbp] 
	\caption{DTI by 2DRQI} \label{alg:dtiby2DRQI}
	\begin{algorithmic}[1]
		\REQUIRE{$m\times m$ stable matrix $\widehat{A}$, 
			${\tt reltol}$, ${\tt tol}$.}
		
		\ENSURE{2D-eigentriplet $(\widehat{\mu},\widehat{\lambda},\widehat{x})$, 
			where $\widehat{\lambda}$ is an estimate of the DTI $\beta(\widehat{A})$,
			and a backward error estimate $\eta_2$.}
		
		\STATE\label{step:DTI1} 
		set $\mu_0$ as the imaginary part of  
		the rightmost eigenvalue of $\widehat{A}$.
		
		\STATE\label{step:DTI2} 
		compute the singular triplet $(u,\lambda_0,v)$ 
		corresponding to the smallest singular value of 
		$\widehat{A}-\mu_0{\tt i}I$.
		
		\STATE\label{step:DTI3} 
		apply the 2DRQI (Algorithm~\ref{alg:2dRQIps})  
		with initial $(\mu_0,\lambda_0, 
		x_0 = \frac{1}{\sqrt{2}}\left[\begin{smallmatrix}
			u\\v \end{smallmatrix}\right])$ 
		and stopping tolerance {\tt tol} to 
		compute an approximate 2D-eigentriplet
		$(\widehat{\mu},\widehat{\lambda},\widehat{x})$
		of $(A, C)$ and the corresponding backward error estimate
		$\eta_2$.    
		
		\STATE validate the computed DTI $\widehat{\lambda}$ 
		with {\tt reltol} (optional).
		
	\end{algorithmic}
\end{algorithm}

Algorithm~\ref{alg:dtiby2DRQI} is an outline of
a 2DRQI-based algorithms for computing $\beta(\widehat{A})$.  
Two remarks are in order. 
\begin{enumerate} 
\item The initial $(\mu_0, \lambda_0, x_0)$ 
(lines 1 and 2) follows the recommendation in \cite{Freitag2011Aaa}, 
and is critical for the success of the computation. 

\item To satisfy the conditions \eqref{eq:xhatcondition0} 
for the approximate 2D-eigenvector $x_k = \left[\begin{smallmatrix}
	x_{k,1}\\
	x_{k,2}
\end{smallmatrix}\right]$, we should add the following steps 
after line~\ref{step:normalization} in the 2DRQI (Algorithm~\ref{alg:2dRQIps}):
\begin{algorithmic}[1]
	\STATE $x_{k+1,1} = \frac{\sqrt{2}}{2}x_{k+1,1}/\|x_{k+1,1}\|$.
	\STATE $x_{k+1,2} = \frac{\sqrt{2}}{2}x_{k+1,2}/\|x_{k+1,2}\|$.
\end{algorithmic}
\end{enumerate} 

For the stopping criterion of the 2DRQI,  
we use a backward error estimate of the computed DTI.
It has been a challenge to properly define the stopping criterion
of iterative methods for computing DTI 
\cite{Freitag2011Aaa,dinverse1999,Kangal2018,ahownear1985}. 
A main reason is that it is meaningless to
define the backward error for a estimated DTI $\widehat{\beta}$ {\em only}. 
Specifically, if a backward error $\widetilde{\eta}$ of $\widehat{\beta}$ 
is defined as
\begin{equation}\label{eq:naiveDefOfDTI}
\widetilde{\eta} =
\inf \left\{\epsilon \mid \exists\,\delta \widehat{A} \,\mbox{such that}\,
\|\delta \widehat{A}\|\leq\epsilon\|\widehat{A}\| \, \mbox{and}\,
\beta(\widehat{A}+\delta \widehat{A})=\widehat{\beta} \right\},
\end{equation} 
then the following proposition shows that the calculation of 
the backward error $\widetilde{\eta}$ is as hard as the calculation 
of the original $\beta(\widehat{A})$. 

\begin{proposition}\label{prop:betaAhatalone}
If $\beta(\widehat{A})>\widehat{\beta}$, then 
$\widetilde{\eta}=\frac{\beta(\widehat{A})-\widehat{\beta}}{\|\widehat{A}\|}$.
\end{proposition}
\begin{proof}
We first prove the inequality 
$\widetilde{\eta}\geq \frac{\beta(\widehat{A})-\widehat{\beta}}{\|\widehat{A}\|}$.
By the definition of $\widetilde{\eta}$, for any $t>0$, there exists 
a matrix $\delta\widehat{A}$ such that 
$\beta(\widehat{A}+\delta\widehat{A})=\widehat{\beta}$ and 
$\|\delta\widehat{A}\|\leq(\widetilde{\eta}+t)\|\widehat{A}\|$. 
By \eqref{eq:47}, there exists $E_{\widehat{\beta}}$ such that 
$\|E_{\widehat{\beta}}\|=\widehat{\beta}$ and 
$(\widehat{A}+\delta\widehat{A})+E_{\widehat{\beta}}$ is unstable. 
Thus $\widehat{A}$ is unstable
under the pertburation $\delta\widehat{A}+E_{\widehat{\beta}}$.
By the definition of $\beta(\widehat{A})$, this implies
\[
\|\delta\widehat{A}+E_{\widehat{\beta}}\|\geq\beta(\widehat{A}).
\]
Thus 
\[
\beta(\widehat{A})\leq \|\delta\widehat{A}+E_{\widehat{\beta}}\|\leq
\|\delta\widehat{A}\|+
\|E_{\widehat{\beta}}\|\leq(\widetilde{\eta}+t)\|\widehat{A}\|+\widehat{\beta}
\]
holds for any $t>0$.  
Let $t\rightarrow0$, then we have
\begin{equation} \label{eq:lowbd} 
\widetilde{\eta}\geq \frac{\beta(\widehat{A})-\widehat{\beta}}{\|\widehat{A}\|}.
\end{equation}

We next prove the inequality 
$\widetilde{\eta}\leq 
\frac{\beta(\widehat{A})-\widehat{\beta}}{\|\widehat{A}\|}$.
By the definition of $\beta(\widehat{A})$, there exists a matrix $E_{\beta}$ 
such that $\|E_{\beta}\|=\beta(\widehat{A})$ and $A+E_{\beta}$ is unstable. 
Let $F = \frac{\beta(\widehat{A})-\widehat{\beta}}{\beta(\widehat{A})}E_{\beta}$. 
Then $\|F\| = \beta(\widehat{A})-\widehat{\beta}<\beta(\widehat{A})$ and 
thus $\widehat{A}+F$ must be stable.
		
Consider $\beta(\widehat{A} + F)$. Since 
$\widehat{A}+F+(E_{\beta}-F) = \widehat{A}+E_{\beta}$ is unstable, we have 
\begin{equation}\label{eq:betaoneside}
\beta(\widehat{A}+F)\leq \|E_{\beta}-F\|=\widehat{\beta}.
\end{equation}
On the other hand, assume there is a matrix $G$ satisfies 
$\widehat{A}+F+G$ is unstable, then by the definition of $\beta(\widehat{A})$,
\[
\beta(\widehat{A})\leq \|F+G\|\leq \beta(\widehat{A})-\widehat{\beta}+\|G\|.
\]
Thus $\|G\|\geq \widehat{\beta}$, which implies
\begin{equation}\label{eq:betaotherside}
\beta(\widehat{A}+F)\geq \widehat{\beta}.
\end{equation}
By \eqref{eq:betaoneside} and \eqref{eq:betaotherside}, we have
\begin{equation}\label{eq:betaAF}
\beta(\widehat{A}+F)=\widehat{\beta}.
\end{equation}
By \eqref{eq:betaAF} and the definition of $\widetilde{\eta}$,
\[
\|\widehat{A}\|\widetilde{\eta} \leq \|F\| = \beta(\widehat{A})-\widehat{\beta}.
 \]
Then we have the inequality
\begin{equation} \label{eq:uppbd} 
\widetilde{\eta} \leq\frac{\beta(\widehat{A})-\widehat{\beta}}{\|\widehat{A}\|}.
\end{equation}
The proposition is then proven by \eqref{eq:lowbd} and \eqref{eq:uppbd}.
\end{proof}

Proposition~\ref{prop:betaAhatalone} implies the exact calculation of 
the backward error 
$\widetilde{\eta}$ could be as hard as the calculation of the original 
$\beta(\widehat{A})$. 
This is analogous to the fact that for eigenvalue problems we do not define the backward error of an approximate eigenvalue only. We consider the backward error of an approximate 
eigenpairs, see e.g.~\cite[Thm.1.3]{Stewart2001Matrix2}.
As an advantage of treating the DTI via the 2DEVP,  
we can establish the notion of the 
backward error for a computed DTI via an approximate
2D-eigentriplet. The resulting backward error estimation 
naturally leads to a reliable stopping criterion 
for an iterative DTI algorithm.

		To that end, let the approximate
		2D eigentriplet $(\widehat{\mu},\widehat{\lambda},\widehat{x})$ 
		of $(A, C)$ be an exact 2D-eigentriplet of 
		structurely-perturbed 2DEVP 
		\begin{subnumcases}{\label{eq:2devp-dti-p}}
			\left[\begin{smallmatrix}
				0&\widehat{A}+\delta\widehat{A}\\
				\widehat{A}^H+\delta\widehat{A}^H&0
			\end{smallmatrix}\right]
			\widehat{x} -\widehat{\mu}C\widehat{x}
			=\widehat{\lambda}\widehat{x}, \\
			\hspace{8.5em}\widehat{x}^HC\widehat{x}  =0, \\
			\hspace{9.3em}\widehat{x}^H\widehat{x}  =1.
		\end{subnumcases}
		for some $\delta \widehat{A}$.  Then we can define a 
		{structure-preserving backward error} of the 2DEVP 
		of the DTI problem as follows:
		\begin{equation}\label{backerr_struct}
			\widehat{\eta}_{{\beta}}(\widehat{\mu},\widehat{\lambda},\widehat{x}) 
			= \inf \left\{ \epsilon \mid \exists\,\delta \widehat A
			\, \mbox{such that}\,  \|\delta \widehat{A}\|\leq\epsilon\|\widehat{A}\|
			\, \mbox{and} \,
			\eqref{eq:2devp-dti-p}\, \mbox{holds}   
			\right\}.  \end{equation}
		We first note that the set in \eqref{backerr_struct} is nonempty
		when the approximate 2D-eigenvector 
		$\widehat{x}=\left[\begin{smallmatrix}
			\widehat{x}_1\\ \widehat{x}_2
		\end{smallmatrix}\right]$ satisfies 
		the conditions \eqref{eq:xhatcondition0}.
		In fact, denote
		${r}=\left[\begin{smallmatrix} {r}_1\\ {r}_2 \end{smallmatrix}\right]$,
		where $r_1 = \widehat{A}\widehat{x}_2-\widehat{\mu} {\tt i}\widehat{x}_2
		-\widehat{\lambda} \widehat{x}_1$
		and $r_2 = \widehat{A}^H\widehat{x}_1+\widehat{\mu} {\tt i}\widehat{x}_1
		-\widehat{\lambda} \widehat{x}_2$. 
		Then it can be shown that the matrix
		\[
		\delta \widehat{A} = 
		\delta \widehat{A}_1 + 
		\delta \widehat{A}_2
		\quad \mbox{with} \quad 
		\delta \widehat{A}_1 = 
		-\left(I-\frac{\widehat{x}_1\widehat{x}_1^H}{\widehat{x}_1^H\widehat{x}_1}
		\right) \frac{r_1 \widehat{x}_2^H}{\widehat{x}_2^H\widehat{x}_2} 
		\quad \mbox{and} \quad
		\delta \widehat{A}_2 = 
		- \frac{\widehat{x}_1 r^H_2}{\widehat{x}_1^H\widehat{x}_1}.
		\]
		is in the set \eqref{backerr_struct}. 
		Meanwhile, we have
		\begin{align}
			\|\delta\widehat{A}\|
			&=\max\limits_{\|z\|=1}
			\left\Vert (\delta \widehat{A}_1  + \delta \widehat{A}_2) z \right\Vert
			=\max\limits_{\|z\|=1}\sqrt{
				\left\Vert  
				\delta \widehat{A}_1 z \right\Vert^2 + 
				\left\Vert  
				\delta \widehat{A}_2 z \right\Vert^2} \nonumber \\ 
			&\leq \sqrt{
				\left\Vert \delta \widehat{A}_1 \right\Vert^2+
				\left\Vert \delta \widehat{A}_2 \right\Vert^2
			} 
			\leq \sqrt{2\|r_1\|^2+2\|r_2\|^2} = \sqrt{2}\|r\|,  \label{eq:deltaAbd}
		\end{align}
		where the second equality lies in the fact that 
		$\delta \widehat{A}_1 z$ is orthogonal to $\delta \widehat{A}_2 z$.
		
		Next we provide an estimate of $\widehat{\eta}_{{\beta}}$.
		Since $\widehat{\eta}_{{\beta}}$ 
		is the backward error of the stuctured 2DEVP \eqref{eq:2devp-dti-p}, 
		the backward error $\eta$ in \eqref{def:backerr} 
		of a generic (unstructured) 2DEVP is 
		{the lower bound of $\widehat{\eta}_{{\beta}}$}:
		\begin{equation}\label{eq:etarelation1}
			\widehat{\eta}_{{\beta}} \geq \eta \geq \eta_1.  
		\end{equation}
		where $\eta_1$ is defined in \eqref{eq:backerr1}.
		On the other hand, by the definition of $\widehat{\eta}_{{\beta}}$
		and \eqref{eq:deltaAbd} , 
		we have {an upper bound of $\widehat{\eta}_{{\beta}}$:}  
		\begin{equation} \label{eq:eta2def}
			\widehat{\eta}_{{\beta}}\leq 
			\eta_2 \equiv \sqrt{2}\frac{\|r\|}{\|\widehat{A}\|}.
		\end{equation} 
		By the facts that $\|\widehat{A}\|=\|A\|$ and $\|C\|=1$, we have
		\begin{equation}\label{eq:eta2reta1}
			\frac{\eta_2}{\eta_1}\leq\frac{\sqrt{2}\frac{\|r\|}{\|\widehat{A}\|}}{\frac{\|r\|}{\|A\|+|\widehat{\mu}|\|C\|}}=\sqrt{2}\left(1+\frac{|\widehat{\mu}|}{\|\widehat{A}\|}\right).
		\end{equation} 
		Combining \eqref{eq:etarelation1}, \eqref{eq:eta2def}, 
		and \eqref{eq:eta2reta1}, we have
		\begin{equation}\label{eq:etarelation3}
			\frac{1}{\sqrt{2}\left(1+\frac{|\widehat{\mu}|}{\|\widehat{A}\|}\right)}
			\eta_2\leq \widehat{\eta}_{{\beta}}\leq\eta_2. 
		\end{equation}
		Therefore $\eta_2$ defined in \eqref{eq:eta2def} can be used as 
		an estimate of $\widehat{\eta}_{{\beta}}$. 
		Consequently, the stopping critera (line~\ref{step:stop}) of the 2DRQI (Algorithm~\ref{alg:2dRQIps}) should be
		\begin{equation}\label{eq:criterion1}
			|\imag(x_{k,1}^Hx_{k,2})|\leq {\tt tol}  
			\quad \mbox{and} \quad 
			\eta_2(\mu_k,\lambda_k,x_k)\leq {\tt tol},
		\end{equation} 
		where {\tt tol} is a prescribed tolerance value.
		In addition, to handle the possible stagnation of the 2DRQI,
		we can also include the following test for possible stagnation:
		\begin{equation} \label{eq:2drqi_stop_ad}
			\eta_2(\mu_k,\lambda_k,x_k)\geq
			\frac{1}{2}\Big(\eta_2(\mu_{k-2},\lambda_{k-2},x_{k-2}) +
			\eta_2(\mu_{k-1},\lambda_{k-1},x_{k-1})\Big).
		\end{equation}
		
For the optional validation step of Algorithm~\ref{alg:dtiby2DRQI}, we know that if the computed $\widehat{\lambda}$ is an acceptable estimate of DTI $\beta(\widehat{A})$, it should satisfy
		\begin{equation} \label{eq:dtiverify} 
			(1-{\tt reltol})\widehat{\lambda} 
			\leq \beta(\widehat{A})\leq \widehat{\lambda}
		\end{equation} 
		for a small {\tt reltol}, where without loss of generality, 
		we assume $\widehat{\lambda} > 0$. Otherwise, according to the symmetric properties of 2D eigenvalues in DTI, we can use $-\widehat{\lambda}$ as an estimate of the DTI $\beta(\widehat{A})$. 
		
		The upper bound of \eqref{eq:dtiverify} naturally holds according to 
		\eqref{eq:beta2D} and $(\widehat{\mu},\widehat{\lambda})$ 
		is a 2D-eigenvalue. For the lower bound of \eqref{eq:dtiverify}, 
		we just need to verify that $H((1-{\tt reltol})\widehat{\lambda})$ 
		has no imaginary eigenvalue. This is based on the following lemma. 
		\begin{lemma}[\cite{BBK1989}] \label{lem:dtitest} 
			For any $\lambda>0$, $\lambda<\beta(\widehat{A})$
			if and only if $G(\lambda)$ has no pure imaginary eigenvalue, 
			where $G(\lambda)$ is an Hamiltonian matrix of the form
			\begin{equation}
				G(\lambda) = \begin{bmatrix}
					\widehat{A}&-\lambda I\\
					\lambda I &-\widehat{A}^H
				\end{bmatrix}.
			\end{equation}
		\end{lemma}
		
		This validiation procedure is the one proposed in 
		\cite{Freitag2011Aaa}. 
		However, it should be noted that checking whether 
		$G((1-{\tt reltol})\widehat{\lambda})$ has no imaginary eigenvalue 
		could be prohibitively expensive for large scale problems.
		Therefore, the validation step is optional in 
		all existing algorithms 
		for computing DTI \cite{Freitag2011Aaa,dinverse1999,Kangal2018}.
		In Section\ref{sec:Experiments}, we will provide a numerical 
		example to show that the 2DRQI outperforms 
		a recently proposed subspace method for the DTI computation.

\section{Numerical examples}\label{sec:Experiments}
In this section, we first present a numerical example 
to illustrate the convergence behavior of the 2DRQI 
(Algorithm \ref{alg:2dRQIps}), and then present
two examples for finding the minmax of two 
Rayleigh quotients (Algorithm~\ref{alg:minmax2RQs}) 
and for computing the DTI (Algorithm~\ref{alg:dtiby2DRQI}).
All algorithms are implemented in MATLAB. 
Numerical experiments are performed on a HP computer
with an Intel(R) Core(TM) 2.60GHz i7-6700HQ CPU and 8GB RAM. 

	\begin{example} \label{eg:synth} 
	{\rm 
		This example illusrates convergence behaviors 
		of the 2DRQI (Algorithm~\ref{alg:2dRQIps}).
		Let us consider the 2DEVP \eqref{2deig} of the matrices
		\[
		A = \begin{bmatrix}
			-0.7&0.01&0.2\\
			0.01&2&0\\
			0.2&0&0 
		\end{bmatrix} 
		\quad \mbox{and} \quad 
		C = \begin{bmatrix}
			0.3&0.01&0.2\\
			0.01&1&0\\
			0.2&0&-1
		\end{bmatrix}.
		\] 
		It can be verified that 
		$(\mu_1,\lambda_1,x_1) = (1,1,\left[\begin{smallmatrix} 
			0\\ \frac{1}{\sqrt{2}} \\ \frac{1}{\sqrt{2}}
		\end{smallmatrix}\right])$ is a 2D-eigentriplet and
		$\lambda_1=1$ is an eigenvalue of $A-\mu_1 C$ with multiplicity 2. 
		In addition, by a brute-force bisection search 
		following the sorted eigencurves 
		$\lambda_1(\mu) \geq \lambda_2(\mu) \geq \lambda_3(\mu)$
		of $A-\mu C$ on the interval $[-1.5,1.5]$, we found additional 
		two 2D-eigenvalues to the machine precision: 
		\begin{align*} 
			(\mu_2,\lambda_2) 
			& = (-0.665101440190437, -0.239801782612878) \\
			(\mu_3,\lambda_3) 
			& = ( -0.145810069397438, -0.744080780565709). 
		\end{align*} 
		Moreover, $\lambda_2$ and $\lambda_3$ are
		the simple eigenvalue of 
		$A-\mu_2 C$ and $A-\mu_3 C$, respectively. 
		The left plot of Figure~\ref{fig:eg1a} are 
		the sorted eigencurves $\lambda_j(\mu)$ for $j = 1, 2, 3$. The maximum 2D-eigenvalue $(\mu_1,\lambda_1) = (1,1)$ is marked in red. The 2D-eigenvalue $(\mu_2,\lambda_2)$ is blue. The minimum 2D-eigenvalue $(\mu_3,\lambda_3)$ is green.
		
		\begin{figure}[tbhp]  
			\centering
			\includegraphics[scale=0.5]{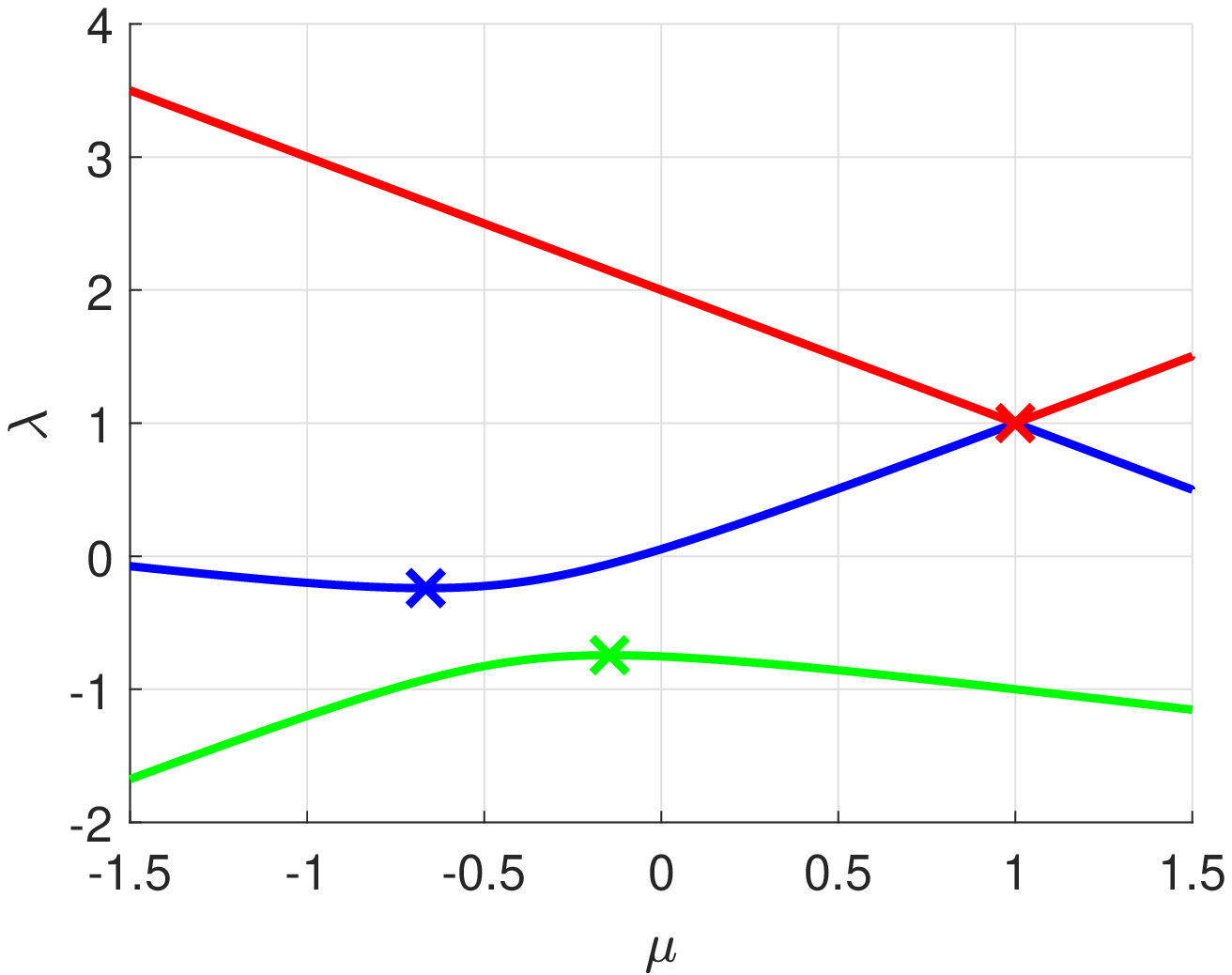}
			\includegraphics[scale=0.5]{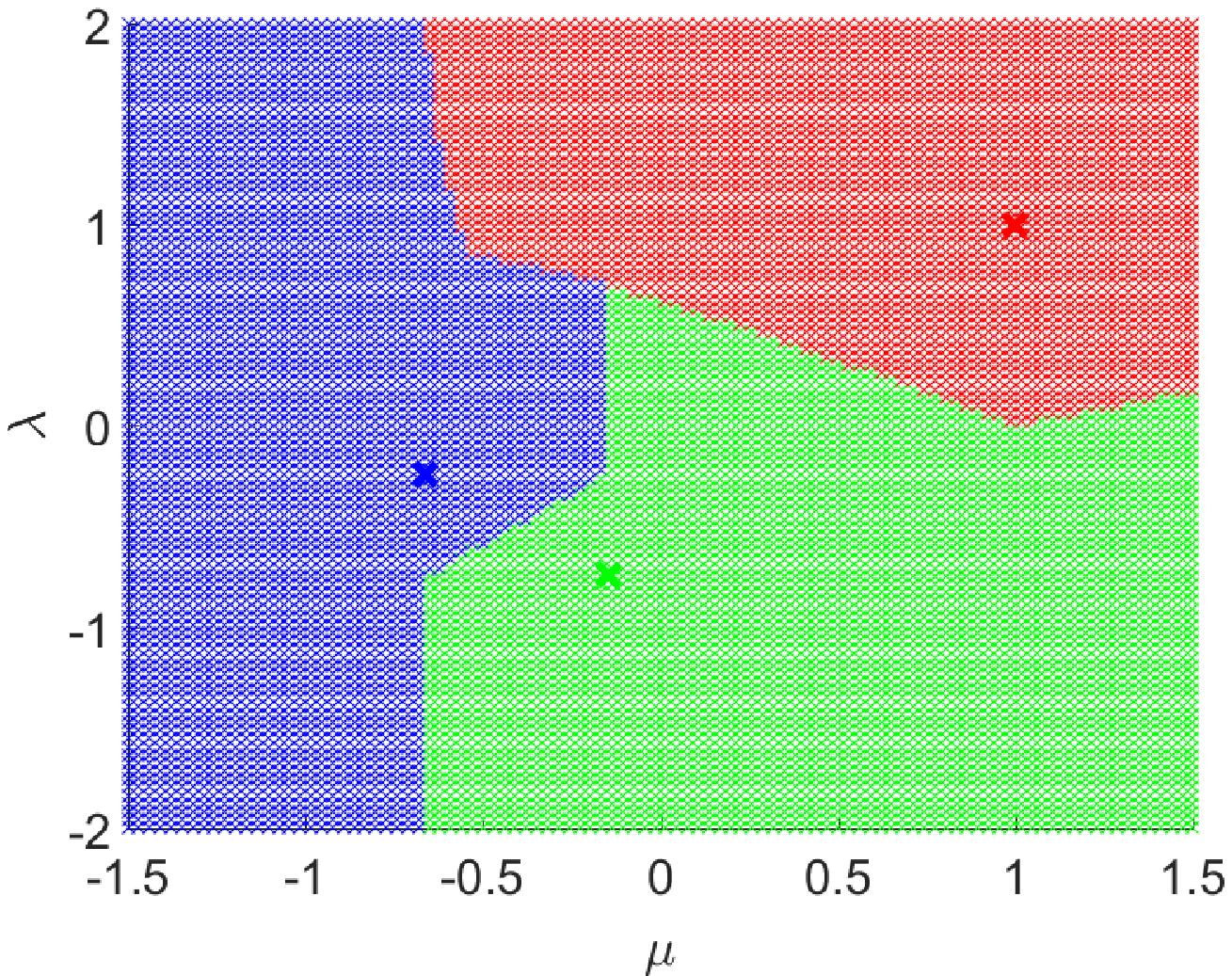}
			\caption{Left: sorted eigencurves and corresponding 2D-eigenvalues
				of $(A, C)$ in Example~\ref{eg:synth}.  
				Right: Computed 2D-eigenvalues with different initials.} \label{fig:eg1a}
		\end{figure}
		
		We use each grid point on the $100\times 100$ mesh of
		the domain $(\mu,\lambda) = [-1.5,1.5]\times[-2,2]$ as 
		an initial $(\mu_0,\lambda_0)$ and
		the vector $x_0$ is generated based on 
		the recommendation of Algorithm~\ref{alg:2dRQIps}. 
		If the 2DRQI with the initial $(\mu_0,\lambda_0,x_0)$ 
		and ${\tt tol} = n\cdot{\tt macheps}$ and maxit = 15, 
		converges to the $i$-th 2D-eigenvalue $(\mu_{i}, \lambda_{i})$, 
		then we use the same color for the initial 
		$(\mu_0, \lambda_0)$ and $(\mu_{i}, \lambda_{i})$.
		The right plot of Figure~\ref{fig:eg1a} shows 
		that the 2DRQI converges to a 2D-eigentriplet for 
		all 10,000 initials $(\mu_0,\lambda_0,x_0)$. 
		
		Table~\ref{table:error} records the convergence history of a sequence 
		$\{(\mu_{3;k}, \lambda_{3;k}, x_{3;k})\}$ 
		to the minimum 2D-eigenvalue $(\mu_3,\lambda_3)$, 
		marked in green in Figure~\ref{fig:eg1a}. 
		We observe that the sequence $\{(\mu_{3;k}, \lambda_{3;k})\}$
		converges quadratically, the matix $C_k$ 
		of the 2DRQ $(A_k, C_k)$ remains to be indefinite and $a_{12,k} \neq 0$. 
		
		\begin{table}[H]
			\caption{Convergence history of $\{(\mu_{3;k},\lambda_{3;k},x_{3;k})\}$ to 
				$(\mu_3,\lambda_3,x_3)$} \label{table:error}
			\centering
			\fontsize{8}{10}\selectfont    
			\begin{tabular}{c|ccccc} \hline
				$k$ & $|\mu_{3;k}-\mu_3|$ & $|\lambda_{3;k}-\lambda_3|$
				& $\eta_1(\mu_{3;k},\lambda_{3;k},x_{3;k})$ 
				& $(c_{1,k},c_{2,k})$ & $|a_{12,k}|$\\ \hline
				0 &  1.6e0    &  8.9e-1    & 4.1e-1  &(-1.0e0,  3.3e-1)  & 2.9e-1   \\
				1 &  2.6e-3   &  8.4e-3    & 7.1e-2  &(-1.0e0,  3.3e-1)  & 2.9e-1   \\
				2 &  2.2e-5   &  1.2e-7    & 2.7e-4  &(-1.0e0,  3.3e-1)  & 2.9e-1   \\
				3 &  6.5e-13  &  1.1e-16   & 2.1e-9  &(-1.0e0,  3.3e-1)  & 2.9e-1   \\ 
				4 &  3.4e-16  &  2.6e-16   & 1.1e-16 &(-1.0e0,  3.3e-1)  & 2.9e-1 \\ \hline
			\end{tabular}\vspace{0cm}
		\end{table}
		
		Table~\ref{table:error2} shows the convergence history of a sequence 
		$\{(\mu_{1;k}, \lambda_{1;k}, x_{1;k})\}$ to the maximum 2D-eigenvalue 
		$(\mu_1,\lambda_1)$, marked in red in Figure~\ref{fig:eg1a}. 
		Note that $\lambda_1$ is an eigenvalue of $A-\mu_1C$ with multiplicity 2. 
		We observe that
		the sequence $\{\mu_{1;k}, \lambda_{1;k}\}$
		converges quadratically and the matix $C_k$
		of the 2DRQ $(A_k, C_k)$ remains to be indefinite. However,
		$a_{12,k}$ approaches to 0.
		
		\begin{table}[H]
			\caption{Convergence history for $\{(\mu_{1;k},\lambda_{1;k},x_{1;k})\}$ to 
				$(\mu^*_1,\lambda^*_1,x^*_1)$.} \label{table:error2}
			\centering 
			\fontsize{8}{10}\selectfont    
			\begin{tabular}{c|ccccc} \hline
				$k$ & $|\mu_{1;k}-\mu_1|$ & $|\lambda_{1;k}-\lambda_1|$ & 
				$\eta_1(\mu_{1;k},\lambda_{1;k},x_{1;k})$ & 
				$(c_{1,k},c_{2,k})$ & $|a_{12,k}|$\\ \hline
				0	&  1.0e0    &  1.0e0     & 3.2e-1  &(-1.0e0,  7.9e-1)  & 9.3e-2    \\
				1	&  3.3e-1   &  4.6e-1    & 3.1e-1  &(-9.6e-1, 9.6e-1 )  & 3.2e-2    \\
				2	&  5.0e-2   &  9.0e-2    & 1.3e-1  &(-1.0e0,  1.0e0 )  & 2.4e-4   \\
				3	&  5.2e-4  &  3.3e-4    & 8.1e-3  &(-1.0e0,  1.0e0 )  & 3.2e-9   \\ 
				4   &  3.8e-10  &  2.2e-11	 & 2.1e-6  &(-1.0e0,  1.0e0 )  & 8.0e-16   \\
				5	&  4.2e-16  &  2.2e-16  & 2.5e-16 &(-1.0e0,  1.0e0 )  & 4.6e-16   \\
				\hline
			\end{tabular}\vspace{0cm}
		\end{table}
		
		In \cite{2DEVPIII}, we will prove 
		that the 2DRQI locally quadratically converges to 
		a 2D-eigentriplet $(\mu_*, \lambda_*, x_*)$ 
		We will see that though the algorithm and local quadratic convergence rate 
		are the same regardless the multiplicity of the eigenvalue $\lambda_*$ 
		of $A - \mu_* C$, convergence analysis needs to be treated 
		differently as indicated by whether $|a_{12,k}|$ approaches to 0. 
} \end{example}


\begin{example}{\rm 
		We use Algorithm~\ref{alg:minmax2RQs} to solve 
		the RQminmax \eqref{prob:minmaxRQ} 
		arising from a MIMO relay precoder design problem in signal 
		communication, and compare with an algorithm proposed in \cite{GH2013}.
		
		The MIMO relay precoder design problem is to minimize
		the total relay power subject to SINR constraints 
		at the receivers~\cite{Chalise2007}. 
		Consider the multi-point to multi-point communication 
		with two sources. The signals $r_o$ after MIMO relay processing 
		and signals $y$ received by destinations are
		\[ 
		r_o=ZH_{\rm up}s+Zn_r 
		\quad \mbox{and} \quad
		y=H_{\rm dl}^HZH_{\rm up}x+H_{\rm dl}^HZn_r+n_d,
		\] 
		where 
		$s$ is the transmit signals of the sources, 
		$n_r$ and $n_d$ are zero-mean circularly symmetric 
		complex Gaussian random variables with variance 
		$\sigma_r^2$ and $\sigma_d^2$.
		$H_{\rm up} = [h_1,\, h_2] \in \mathbb{C}^{m\times 2}$
		denotes channels between two sources and antennas, 
		$H_{\rm dl}= [g_1, \, g_2] \in \mathbb{C}^{m\times 2}$
		denotes channels between antennas and two destinations,
		$m$ is the number of antennas at the relay.
		$Z\in\mathbb{C}^{m\times m}$ is the MIMO relay processing matrix 
		to be designed. 
		Under the assumption that the source transimit signals $s$ are 
		zero-mean, statistically independent with the unit power, 
		the goal of the MIMO relay precoder design is to minimize the relay power 
		while maintaining SINR no less than a prescribed threshold $\gamma_{\rm th}$. 
		
After some algebraic manipulations, 
the MIMO precoder relay design problem becomes solving the following 
homogeneous quadratic constrained programming (HQCQP):
		\begin{equation} \label{prob_P1}
			\min_{u} u^HTu  \quad \text{s.t.} \quad
			u^HP_iu+1 \le 0 \quad \text{for} \quad i=1,2,
		\end{equation}
		where 
		$u = \vvec(Z)$ is a vector of length $n=m^2$, 
		$T = \widehat{F}_0 \otimes I$, 
		$P_1 = \widehat{F}_1 \otimes g_1 g^H_1$ and 
		$P_2 = \widehat{F}_2 \otimes g_2 g^H_2$ are
		of dimensione $n=m^2$, with
\begin{align*} 
\widehat{F}_0 & = \overline{h}_1h_1^T+\overline{h}_2h_2^T+\sigma_r^2I, \\  
\widehat{F}_1 & = \frac{1}{\gamma_{\rm th}\sigma_d^2}
		\left(\gamma_{\rm th}\overline{h}_2h_2^T+
		\gamma_{\rm th}\sigma_r^2I-\overline{h}_1h_1^T\right), \\  
\widehat{F}_2 & = \frac{1}{\gamma_{\rm th}\sigma_d^2}
		\left( \gamma_{\rm th}\overline{h}_1h_1^T+
		\gamma_{\rm th}\sigma_r^2I-\overline{h}_2h_2^T \right).
\end{align*} 
Note that $\widehat{F}_0$ and $\widehat{F}_i$ are $m\times m$ 
Hermitian matrices with $\widehat{F}_0$ positive definite.
Gaurav and Hari \cite{GH2013} show that 
the HQCQP \eqref{prob_P1} is equivalent to
the RQminmax~\eqref{prob:minmaxRQ} of the matrices 
\begin{equation} \label{eq:ACminmaxRQ1} 
A = S^{H}P_1S = F_1\otimes g_1g_1^H, \quad
B = S^{H}P_1S = F_2\otimes g_2g_2^H,
\end{equation} 
where $S=T^{-\frac{1}{2}}$ is the square root of $T^{-1}$,
$F_1 = \widehat{F}_0^{-\frac{1}{2}}\widehat{F}_1\widehat{F}_0^{-\frac{1}{2}}$ 
and 
$F_2 = \widehat{F}_0^{-\frac{1}{2}}\widehat{F}_2\widehat{F}_0^{-\frac{1}{2}}$. 
		We note that by exploiting the structure of $A$ and $B$, 
		the matrix-vector multiplications $Ax$ and $Bx$ 
		can be performed efficiently. 
		
		Algorithm~\ref{alg:minmax2RQs} 
		first checks the Cases-I and II of 
		the RQminmax~\eqref{prob:minmaxRQ}
		described in Theorem~\ref{thm:classfication1}  
		for possible early exit. 
		Then it uses a combination of the 2DRQI and the bisection search
		to find an optimizer $\mu_{*}^{(\rm RQI)}$ of 
		the EVopt~\eqref{eq:EVopt0} for the general Case-III.
		
		A dichotomous method is
		proposed in \cite{GH2013} for solving the EVopt~\eqref{eq:EVopt0}.
		Starting from a search interval $[a,b]$ containing the global maximum 
		of the concave function $g(\mu) = \lambda_{\min}(A-\mu C)$, 
		where $C = A - B$, the dichotomous method 
		compares $g(a)$, $g(b)$, $g(\frac{a+b}{2}-\epsilon_r)$
		and $g(\frac{a+b}{2}+\epsilon_r)$ 
		for a small scalar $\epsilon_r$, and then 
		by using the concavity of $g(\mu)$, replaces $a$ 
		with $\frac{a+b}{2}-\epsilon_r$, or 
		$b$ with $\frac{a+b}{2}+\epsilon_r$ for the next iteration. 
		When the search interval width $b-a$ is less than 
		a prescribed tolerance {\tt tol}, it returns an approximate 
		optimal value $\mu_{*}^{(\rm Dich)} = \frac{a+b}{2}$.  
		
		For numerical experiments described in \cite{GH2013},
		$H_{\rm up}$ and $H_{\rm dl}$ are complex Gaussian random matrices.
		The SINR is set to 3dB and noise variances are set to -10dB, i.e., 
		$\gamma_{\rm th}=10^{\frac{3}{10}}$, and $\sigma_d^2=\sigma_r^2=10^{-1}$.
		
		We observed that the optimizers of
		the EVopt \eqref{eq:EVopt0} on the interval $[0,1]$ 
		computed by the dichotomous method with 
		{\tt tol} = {\tt 1e-8} 
		and Algorithm~\ref{alg:minmax2RQs}
		with {\tt backtol} = $n\epsilon$ for the 2DRQI
		and {\tt reltol} = {\tt 1e-8} agree up to 8 sigificant digits:
		${|\mu_{*}^{(\rm RQI)} - \mu_{*}^{(\rm Dich)}|}/{|\mu_{*}^{(\rm Dich)}|}
		\leq {\tt 1e-8}$ for 20 runs of each of 
		dimensions $n=10^2, 100^2, 200^2, 400^2$. 
		
		The third column of Table~\ref{table:dist} reports 		the average runtime (in seconds) of 100 runs of the dichotomous method for finding the optimizer of the EVopt \eqref{eq:EVopt0} on the interval $[0,1]$ with the accuracy {\tt tol} = {\tt 1e-4}. The fifth column of Table~\ref{table:dist} reports the average runtime of 100 runs of Algorithm~\ref{alg:minmax2RQs} with {\tt reltol} = {\tt 1e-8} and {\tt backtol} = $n\epsilon$, excluding the lines~\ref{step:A1} and~\ref{step:A2} of Algorithm~\ref{alg:minmax2RQs} for checking the Cases-I and II. 
		
		\begin{table}[htbp]  
			\caption{
				Performance of the dichotomous method and Algorithm~\ref{alg:minmax2RQs} 
				for solving the EVopt \eqref{eq:EVopt0}. 
			} \label{table:dist}
			\centering
			\begin{tabular}{|c||cc||cc|} \hline
				& \multicolumn{2}{c||}{Dichotomous method}  
				& \multicolumn{2}{c|}{Algorithm~\ref{alg:minmax2RQs}} \\ 
				$n = m^2$ & \multicolumn{1}{c}{niter}  & 
				\multicolumn{1}{c||}{runtime} & \multicolumn{1}{c}{niter} & 
				\multicolumn{1}{c|}{runtime} \\ \hline
				$10^2$ & 15   & 0.11 & {3.1}  & {0.026}  \\ 
				$100^2$ & 15  & {1.2}  & 2.6 & {0.19}   \\ 
				$200^2$ & 15  & {4.6}   & 2.4 & {0.57}  \\ 
				$400^2$ & 15  & {29}    & 2.1 & {3.6}  \\ \hline	
			\end{tabular}
		\end{table}
		
		The significant performance gain of Algorithm~\ref{alg:minmax2RQs}
		in speed is due to the reduction of the number 
		of iterations shown in the ``niter'' columns of Table~\ref{table:dist}, 
		and the fact that each iteration of the dichotomous method 
		needs to solve two eigenvalue problems of $A - \mu_i C$ for 
		computing $g(\mu_i) = \lambda_{\min}(A-\mu_i C)$, where we use
		the sparse eigensolver {\tt eigs}.   
		In contrast, each iteration of Algorithm~\ref{alg:minmax2RQs}, 
		calls the 2DRQI (Algorithm~\ref{alg:2dRQIps}) once,  
		which in turn only needs to solve the linear system~\eqref{2DRQIaugmen},
		where we use the linear solver {\tt gmres}. 
}\end{example} 

\begin{example} \label{eg:case4} 
	{\rm
		The  purpose of this example is to show that 
		Algorithm~\ref{alg:dtiby2DRQI} 
		is more efficient than 
		recently proposed subspace method \cite{Kangal2018}
		for large scale DTI computation. 
		
		An $n\times n$ Orr-Sommerfeld matrix from
		finite difference discretization of
		the Orr-Sommerfeld operator for planar Poiseuille flow
		is of the form\footnote{The formulation
			in \cite{dinverse1999,Kangal2018} has some typos.}:
		\[
		\widehat{A}_n = L_n^{-1}B_n,
		\]
		where 
		$L_n = (1/h^2){\rm tridiag}(1,-(2+h^2),1)$,
		$B_n = \frac{1}{\mathcal{R}_e}L_n^2-{\tt i}(U_nL_n+2I)$
		and $U_n = \diag(1-u_1^2,\cdots,1-u_n^2)$. 
		$h= 2/(n+1)$ is the stepsize of discretization, $u_k = -1+kh$,
		$\mathcal{R}_e$ is the Reynolds number 
		($\mathcal{R}_e = 1000$ in numerical experiments) 
		and ${\tt i} = \sqrt{-1}$.
		The stability of the Orr-Sommerfeld matrices has been extensively 
		studied \cite{drazin2004,Malyshev1999,Reddy1993}. 
		It is known that the eigenvalues of Orr-Sommerfeld matrices are 
		highly sensitive to perturbations. The DTI is an important measure 
		of the stability under 
		perturbation~\cite{Freitag2011Aaa,dinverse1999,Kangal2018}.
		
		To apply Algorithm~\ref{alg:dtiby2DRQI}  for
		computing the DTI of $\widehat{A}_n$,
		we need to solve the linear equation \eqref{2DRQIaugmen} 
		in the 2DRQI. For computational efficiency, we first transform 
		the Jacobian $J(\mu_k,\lambda_k, x_k)$ into 
		a banded arrow matrix~\cite[p.\,86]{chen2005} through 
		a permutation, and then apply a Schur complement technique 
		\cite[p.406]{nocedal2006}.
	
	For the initial $(\mu_0, \lambda_0, x_0)$ of the 2DRQI, 
	we apply the Cayley-Arnoldi algorithm with complex shift
	for computing $\mu_0$~\cite{Meerbergen1996}, and then 
	use MATLAB's {\tt svds} to compute
	the smallest singular triplet of the matrix $\widehat{A}-\mu_0{\tt i}I$.
	We set {\tt reltol} = {\tt 1e-9} and ${\tt tol} = n\epsilon$.
	
	
	A subspace method \cite{Kangal2018} for eigenvalue optimization 
	is recently applied for computing DTI $\beta(\widehat{A}_n)$ 
	based on the singular value minimization:
	\begin{equation}\label{eq:subspace:singularvalue}
		\beta(\widehat{A}_n) 
		= \min_{\mu \in \mathbb{R}} \sigma_{\min}(\widehat{A}_n-\mu {\tt i}I).
	\end{equation}
	With a prescribed search interval $[a,b]$ and 
	an initial $\mu_0\in[a,b]$, the subspace method 
	first computes $\sigma_{\min}(\widehat{A}_n-\mu_0 {\tt i}I)$ and
	the corresponding right singular vector $v_0$ and 
	then sets the initial projection subspace $V_0 = v_0$. 
	At the $k$-th iteration for $k \geq 1$, 
	the subspace method projects 
	the minimization~\eqref{eq:subspace:singularvalue} onto 
	the subspace $V_{k-1}$ and solves the reduced problem: 
	\begin{equation} \label{eq:dtisubp} 
		\sigma_{\min}^{(k)} =
		\min_{\mu\in [a,b]}\sigma_{\min}(\widehat{A}_nV_{k-1}-\mu {\tt i}V_{k-1}).
	\end{equation} 
	With a minimizer $\mu_k$ of the reduced problem \eqref{eq:dtisubp},
	the subspace method computes 
	$\sigma_{\min}(\widehat{A}_n-\mu_k {\tt i}I)$
	and the corresponding right singular vector $v_k$, and then
	updates the projection subspace
	$V_k = \mbox{Orth}(v_{k-1},v_k)$.
	The iteration terminates when
	$\sigma_{\min}^{(k-1)}-\sigma_{\min}^{(k)}< {\tt tol}$
	for a prescribed {\tt tol}, or the number of iterations exceeds $\sqrt{n}$.
	
	{\tt leigopt} is an implementation of the 
	subspace method in MATLAB~\cite{Kangal2018}\footnote{
		\url{http://home.ku.edu.tr/\textasciitilde emengi/software/leigopt}, 
		downloaded on October 2, 2021. 
	}. 
	To improve computational efficiency, the following minor modifications 
	are made in {\tt leigopt}. 
	(1) We set the dimension of the projection subspace {\tt opts.p = 20} in 
	{\tt eigs} or {\tt svds}, instead of {\tt round(sqrt($n$))} used 
	in {\tt leigopt}. It is observed significant reduction in computational cost. 
	(2) {\tt leigopt} uses {\tt eigopt}, a quadratic supporting functions based 
	method \cite{MYK2014}, to solve the reduced problem~\eqref{eq:dtisubp}. 
	For the Orr-Sommerfeld matrices, {\tt eigopt} 
	is too time consuming.  Instead, we use a modified 
	Boyd-Balakrishnan method \cite{boyd1990}\footnote{This strategy is also 
		recommended by Mengi, one of the authors of {\tt leigopt} 
		in a private communication.}. 
	For numerical experiments, the search interval of {\tt leigopt} is set to 
	$[a,b]= [-60,60]$, the initial $\mu_0=0$ and the tolerance 
	{\tt tol} = {\tt 1e-12}. 
	
	Table~\ref{table3} shows the performance of Algorithm~\ref{alg:dtiby2DRQI} 
	and the subspace method.  The runtime of Algorithm~\ref{alg:dtiby2DRQI} 
	is written as $t_1+t_2$ with $t_1$ for calculating the rightmost 
	eigenvalue of $\widehat{A}_n$ and the singular triplet of 
	$\widehat{A}_n - \mu_0 {\tt i} I$ (i.e., lines 1 and 2 of 
	Algorithm~\ref{alg:dtiby2DRQI}), and $t_2$ for the rest of calculation. 
	The runtime of the subspace method is written 
	as $t_{\rm all} (t_{\rm sub})$ with 
	$t_{\rm all}$ for the total time and $t_{\rm sub}$ 
	for solving the subproblems \eqref{eq:dtisubp}. 
	
	We observe that the computed $\widehat{\beta}(\widehat{A}_n)$ by
	two algorithms agrees from 4 to 8 significant digits. 
	However, Algorithm~\ref{alg:dtiby2DRQI} uses no more than half of 
	the runtime of the subspace method. The speedup of Algorithm~\ref{alg:dtiby2DRQI} comes from two-fold. Algorithm~\ref{alg:dtiby2DRQI} uses less iterative steps. The major cost of the subspace method is on computing the right singular vector $v_k$ corresponding to $\sigma_{\min}(\widehat{A}-\mu_k{\tt i}I)$. In contrast, in Algorithm~\ref{alg:dtiby2DRQI}, we only need to solve a linear equation of the form~\eqref{2DRQIaugmen} in each iteration of 2DRQI (Algorithm~\ref{alg:2dRQIps}).
	
	\begin{table}[htbp]   
		\caption{DTI computation by the subspace method and 
			Algorithm~\ref{alg:dtiby2DRQI}.} \label{table3}
		\begin{center} 
			\begin{tabular}{|l||lll||lll|} \hline
				&\multicolumn{3}{c||}{The subspace method}     
				& \multicolumn{3}{c|}{Algorithm~\ref{alg:dtiby2DRQI}}   \\  
				$n$ &  \multicolumn{1}{c}{niter} & \multicolumn{1}{c}{runtime} 
				& \multicolumn{1}{c||}{$\widehat{\beta}(\widehat{A}_n)$} & 
				\multicolumn{1}{c}{niter} 
				& \multicolumn{1}{c}{runtime} & 
				\multicolumn{1}{c|}{$\widehat{\beta}(\widehat{A}_n)$} \\ \hline
				1000 &  9.7 &  0.16(0.013) & 1.97789572460e-3 & 
				5.8 &  $0.025 + 0.032$ & 1.9778957275e-3 \\ 
				4000 &  9.7 &  0.44(0.017) & 1.97809438632e-3 & 
				4.9 & $0.062 + 0.095$ & 1.9780964583e-3 \\ 
				16000 & 8.9 &  1.53(0.035) & 1.93786346536e-3 & 
				4.8 & $0.25\ \, + 0.38$ &  1.9376706543e-3 \\ \hline	
			\end{tabular}
		\end{center} 
	\end{table}
	
	We note that the validation step for computed 
	$\widehat{\beta}(\widehat{A}_n)$ by  
	Algorithm~\ref{alg:dtiby2DRQI} and the subspace method 
	is not reported in Table~\ref{table3}. 
	For the matrix size $n=1000$, it is verified that both algorithms 
	pass the validation procedure described in Section~\ref{sec:dti}. 
	Although there exists an algorithm \cite{csparse2006} for checking whether 
	$G(\lambda)$ defined in Lemma~\ref{lem:dtitest} has 
	pure imaginary eigenvalues, it would be too expensive for large matrix
	sizes.  As a common practice of existing 
	algorithms \cite{Freitag2011Aaa,dinverse1999,Kangal2018,ahownear1985}, 
	there is no validation procedure for large scale DTI calculation. 
} \end{example}


\section{Concluding remarks}\label{sec:conclusion}
Based on the theoretical results presented in Part I of this paper~\cite{2DEVPI},  
we devised an RQI-like algorithm, 2DRQI in short, for solving
the 2DEVP \eqref{2deig}. The computational kernel of the 2DRQI 
is on solving a linear systems of equation. The efficiency
of the 2DRQI is demonstrated for solving large scale 2DEVP arising from
the minmax problem of two Rayleigh quotients and
the computation of the distance to instability 
of a stable matrix. A rigorous convergence analysis 
of the proposed 2DRQI will be presented in the third part of this paper. 


\bibliographystyle{siamplain}
\bibliography{2devp}
\end{document}